\documentclass{article}

\usepackage{fullpage}

\usepackage{graphicx}
\usepackage{subcaption}
\usepackage{amsmath,amssymb,amsfonts,amsthm}
\usepackage{mathalfa}
\usepackage{tikz}
\usetikzlibrary{fadings}
\tikzfading[name=fade out,
            inner color=transparent!20,
            outer color=transparent!90]
\tikzfading[name=fade equal,
            inner color=transparent!30,
            outer color=transparent!30]
\usetikzlibrary{circuits.ee.IEC}
\usepackage[american]{circuitikz}
\usetikzlibrary{calc}
\usepackage{color}

\newenvironment{myproof}[1]{\smallskip\noindent{\sc Proof #1.}}%
        {\hspace*{\fill}$\Box$\par}

\usepackage{enumitem}
\newlist{ass}{enumerate}{1}
\setlist[ass]{label = {\bf (A\arabic*)}, resume}
\newlist{assFP}{enumerate}{1}
\setlist[assFP]{label = {\bf (A\arabic*-P)}, resume}
\newlist{assBV}{enumerate}{1}
\setlist[assBV]{label = {\bf (A\arabic*-BV)}, resume}
\newlist{hyp}{enumerate}{1}
\setlist[hyp]{label = {(H\arabic*)}, resume}
\newlist{obj}{enumerate}{1}
\setlist[obj]{label = {\bf (O\arabic*)}, resume}

\theoremstyle{plain}
\newtheorem{lem}{Lemma}
\newtheorem{prop}{Proposition}
\newtheorem{thm}{Theorem}
\newtheorem{cor}{Corollary}
\theoremstyle{definition}
\newtheorem{defn}{Definition}
\newtheorem{exam}{Example}
\theoremstyle{remark}
\newtheorem{rem}{Remark}

\newcommand{\N}{\mathbb{N}}

\newcommand{\R}{\mathbb{R}}

\newcommand{\cF}{\mathcal{F}}

\newcommand{\cH}{\mathcal{H}}

\newcommand{\cK}{\mathcal{K}}
\newcommand{\cL}{\mathcal{L}}
\newcommand{\cM}{\mathcal{M}}
\newcommand{\cN}{\mathcal{N}}

\newcommand{\cP}{\mathcal{P}}

\newcommand{\cS}{\mathcal{S}}
\newcommand{\cT}{\mathcal{T}}

\newcommand{\argmin}{\operatornamewithlimits{arg\,min}}
\DeclareMathOperator{\rank}{rank}

\DeclareMathOperator{\svm}{svm}
\DeclareMathOperator{\Haus}{Haus}
\DeclareMathOperator{\im}{im}
\DeclareMathOperator{\cl}{cl}
\DeclareMathOperator{\gph}{graph}
\DeclareMathOperator{\dom}{dom}
\DeclareMathOperator{\rint}{rint}
\DeclareMathOperator{\inn}{int}
\DeclareMathOperator{\rge}{rge}

\DeclareMathOperator{\col}{col}

\newcommand{\reg}{\text{\normalfont reg}}

\colorlet{Darkred}{red!50!black}
\colorlet{Darkgreen}{green!50!black}
\xdefinecolor{UIUCblue}{rgb}{0,0.24,0.49} 
\xdefinecolor{UIUCorange}{rgb}{0.96,0.5,0.14} 

\newcommand{\footremember}[2]{%
   \footnote{#2}
    \newcounter{#1}
    \setcounter{#1}{\value{footnote}}%
}

\title{Well-Posedness and Output Regulation for Implicit Time-Varying Evolution Variational Inequalities\footremember{conf}{A preliminary version of the  results was presented at 53rd IEEE Conference on Decision and Control, Los Angeles, USA in December 2014. See \cite{TanwBrog14c}.}}
\author{%
    Aneel Tanwani\footremember{laas}{LAAS--CNRS, University of Toulouse, CNRS, 7 Avenue du Colonel Roche, 31400 Toulouse, France.
    Email: {\tt aneel.tanwani@laas.fr}. His work is partially supported by the project ANR-17-CE40-0019-01.}%
    \and Bernard Brogliato\footremember{inria}{INRIA Grenoble Rh\^one-Alpes, Universit\'e Grenoble-Alpes, 655 avenue de l'Europe, 38334 Saint-Ismier, France.
    Email: {\tt bernard.brogliato@inria.fr}
}%
    \and Christophe Prieur\footremember{gipsa}{Universit\'e Grenoble-Alpes, CNRS, GIPSA-lab, F-38000 Grenoble, France.
    Email: {\tt christophe.prieur@gipsa-lab.fr.} His work has been partially supported by the LabEx PERSYVAL-Lab (ANR-11-61 LABX-0025-01).
}%
}
\date{}

\begin{document}
\maketitle

%
%

\begin{abstract}
A class of evolution variational inequalities (EVIs), which comprises ordinary differential equations (ODEs) coupled with variational inequalities (VIs) associated with time-varying set-valued mappings, is proposed in this paper.
We first study the conditions for existence and uniqueness of solutions. The central idea behind the proof is to rewrite the system dynamics as a differential inclusion which can be decomposed into a single-valued Lipschitz map, and a time-dependent maximal monotone operator. Regularity assumptions on the set-valued mapping determine the regularity of the resulting solutions. Complementarity systems with time-dependence are studied as a particular case. We then use this result to study the problem of designing state feedback control laws for output regulation in systems described by EVIs.
The derivation of control laws for output regulation is based on the use of internal model principle, and two cases are treated: First, a static feedback control law is derived when full state feedback is available; In the second case, only the error to be regulated is assumed to be available for measurement and a dynamic compensator is designed. As applications, we demonstrate how control input resulting from the solution of a variational inequality results in regulating the output of the system while maintaining polyhedral state constraints. Another application is seen in designing control inputs for regulation in power converters.
\end{abstract}




\section{Introduction}
Nonsmooth dynamical systems provide a modeling framework to describe discontinuous changes in vector fields, and the state trajectories. The evolution of the trajectories of such systems is often described via set-valued mappings and leads to several different formalisms which have been developed somewhat independently in the literature. As examples, one may see the books on viability theory \cite{AubiCell84}, Fillipov systems \cite{Fili88}, sweeping processes \cite{Mont93}, nonsmooth mechanics \cite{Brog16}, and the articles on closed-convex processes \cite{Goeb14}, differential variational inequalities \cite{PangStew08}, or an overview of discontinuous dynamical systems \cite{Cortes08}.
Evolution variational inequalities (EVIs) provide a mathematical framework to model evolution of state trajectories which, in addition to differential equations, satisfy some algebraic relations as well.
Roughly speaking, EVIs comprise an ordinary differential equation (ODE) to describe the motion of the state variable, and a {\em variational inequality} (VI) that expresses the constraints, and relations that must be satisfied by the state variable.
VIs are most commonly encountered in optimization, and different classes of EVIs have found applications in modeling of electrical circuits with nonsmooth devices, and mechanical systems with impacts.

Building on our recent conference papers \cite{TanwBrog14b, TanwBrog14c}, this article studies the conditions for existence and uniqueness of solutions for a certain class of EVIs with time-varying set-valued map. These results are also used for synthesizing controllers which regulate the output of the system to a prescribed output trajectory. The applications of these results are investigated in the context of designing inputs for viability of convex sets, and regulation problems in electrical circuits.

In the standard formulation of VIs \cite{FaccPang03}, we consider a mapping, say $F:\R^{d_s} \rightarrow \R^{d_s}$, and are interested in finding a vector $v \in \R^{d_s}$ in the domain of $F$, such that  the inequality
\[
\langle F(v), v' - v \rangle \ge 0, \quad \forall \, v' \in \cS,
\]
holds for some closed, convex, and stationary set $\cS \subseteq \R^{d_s}$.
For the dynamical systems considered in this paper, we couple such VIs with ODEs, and moreover, we also let the set $\cS$ be time-varying.
The added difficulty treated in this paper deals with the case when the variable $v$ itself depends on the function $F$ (see \eqref{eq:sysDVI} in Section~\ref{sec:prelim} for exact description), which motivates the term `implicit' in the title.
From the geometric viewpoint, a VI can be equivalently expressed as finding $v$ such that $-F(v)$ is a vector in the normal cone to $\cS$ at $v$.
The normal cone defines a set-valued operator on the set $\cS$, and thus, different tools from the theory of differential inclusions could be used to study the solutions of EVIs.
In our work, we use the fact that the set-valued operator defined by the normal cone is in fact {\em maximal monotone}, and build on the ground work laid by Br\'ezis in his book \cite{Brez73} on solutions of differential inclusions with maximal monotone operators; one may also consult the recent tutorial-like article \cite{PeypSori10}.
For the class of EVIs studied in this paper, motivated primarily by modeling of physical systems, the set-valued maximal monotone operator does not appear explicitly in the description of the differential equation/inclusion. The multivalued component, which is also time-varying in our problem formulation, is only defined implicitly through an inequality which resembles a VI.
Conditions on system data are imposed so that this multivalued operator could be decomposed as a single-valued Lipschitz continuous function, and a maximal monotone set-valued operator.
This decomposition is the key element in establishing the proof of well-posedness. Several particular cases where $\cS$ is the positive orthant (complementarity systems), and the regularity of $\cS$ is imposed by a function of bounded variation, are also treated.

Our approach extends the work of \cite{BrogGoel11, BrogGoel13} on the existence of solutions by relaxing certain assumptions on the matrices that describe the VI.
Also, in the literature, the papers \cite{Brog04, CamlIann14, CamlSchu15} study existence and uniqueness of solutions for systems with linear vector fields coupled with maximal monotone operators.
Making an elegant connection between the notion of passivity (which appears primarily in control theory) and maximal monotonicity of set-valued maps, it is shown that if the matrices describing the set-valued operator and the vector field satisfy certain linear matrix inequalities (LMIs), then there exists a unique solution to the system under consideration. The paper \cite{Brog04} is a particular case of \cite{CamlSchu15}, where the latter also takes into account the implicit description of the set-valued operator.
The results of this paper also generalize the results of \cite{Brog04, CamlSchu15} by considering nonlinear vector fields in system description, and nonstationary sets $\cS$.
The proof worked out in this case is also completely different than the work of \cite{CamlSchu15}.

The second main contribution of this paper comes in applying these results to study the control design problem of output regulation. The system dynamics are driven by a control input which can be chosen as a function of measured states to drive the output of the system to track a prescribed function asymptotically.
Our approach is based on the pioneering work of \cite {Francis77} who proposed simple algebraic criteria for output regulation in multivariable linear time-invariant (LTI) systems.
In \cite{Francis77}, it is assumed that the controlled plant is driven by the output of an LTI exosystem that models the dynamics of the reference trajectories and/or disturbances.
Intuitively speaking, the proposed control input that achieves the output regulation comprises a feedback component to make the closed-loop dynamics stable and an additional open-loop component that shapes the steady state of the plant.
The derivation of the open-loop component of the control input requires the exact knowledge of the exosystem dynamics, and hence the approach is termed as {\em internal model principle}.
A large amount of work has followed on the problem of output regulation using internal model principle in more general contexts and this paper could be seen as extending this principle for a class of EVIs.

In our results on output regulation, we will restrict ourselves to vector fields which are linear in state and input.
The sets considered in the description of VIs for the exosystem and the plant are assumed to be the same but the mappings used to describe the relations could be different.
We derive sufficient conditions under which there exists a control input that achieves output regulation while maintaining state constraints.
In addition to the classical regulator synthesis equations, additional conditions are needed in our work to generate a dissipative relation between the multivalued part and the output regulation error.
These additional conditions also guarantee that the closed-loop system is well-posed, that is, it admits a unique solution which is an important consideration for designing controllers for such class of systems.
We study two cases for control synthesis depending on how much information is available to the controller. In the first case, it is assumed that the entire states of the plant and the exosystem are available and thus, a static controller is designed to achieve output regulation.
In the second case, it is assumed that only the regulation error (which needs to converge to zero) is available and in that case a dynamic compensator is designed.

Moving from control design to applications, we consider two particular case studies. First, we consider reference signals generated by linear complementarity systems (LCS), and design a control input for an LTI system that not only achieves the output regulation but also forces the state-trajectory to evolve within a predefined time-varying polyhedral subset of the state space.
Secondly, we consider electrical circuits with nonsmooth devices which can be modeled using the framework of EVIs. The problem of designing regulators for such systems is of wide interest in engineering community, and here we present how our control design algorithms can be applied to solve some of these problems. An example is provided with each of the case studies with simulation results. In the broader context, these applications relate to solving regulation problems in the presence of state constraints. In the literature, we see an example of control of positive systems in this direction \cite{Rant15}. In positive systems, the {\em smooth} vector fields are such that the resulting trajectory evolves in the positive orthant. Here, even if we choose the constraint set to be the positive orthant, discontinuities in the vector fields arise when the state is at the boundary of the constraint set, so that the resulting trajectory is ``forced to'' satisfy the constraints by introducing discontinuities.


The rest of the paper is organized as follows: In Section~\ref{sec:prelim}, we define the system class, formulate the problem of output regulation that we consider, and introduce some basic results from convex analysis. These results are used to develop a result on existence and uniqueness of solutions for the proposed system class in Section~\ref{sec:sol}.
The design of static state feedback is considered in Section~\ref{sec:static}, followed by the design of a dynamic compensator in Section~\ref{sec:dyn}.
We discuss the applications of our results in Section~\ref{sec:app}.

\section{Preliminaries}\label{sec:prelim}
\subsection{Evolution Variational Inequalities}
Consider a set-valued mapping $\cS:[0,\infty) \rightrightarrows \R^{d_s}$, and assume that $\cS(t)$ is closed, convex, and nonempty, for each $t \ge 0$.
The class of EVIs considered in this paper is described as follows:
\begin{subequations}\label{eq:sysDVI}
\begin{align}
\dot x(t) &= f(t,x(t)) + G \eta(t) \label{eq:sysDVIa} \\
v(t) & = Hx(t) + J \eta(t), \quad v(t) \in \cS(t), \label{eq:sysDVIb} \\
\langle v'- & v(t), \eta(t)\rangle  \ge 0, \quad \forall \, v'\in \cS(t) \label{eq:sysDVIc}.
\end{align}
\end{subequations}
In the above equation $x(t) \in \R^n$ denotes the state, $\eta(t), v(t) \in \R^{d_s}$ are vectors, the vector field $f:[0,\infty)\times \R^n \rightarrow \R^n$ is Lebesgue integrable in time (the first argument), and globally Lipschitz in the state variable (with possibly time-varying Lipschitz modulus). Moreover, $G \in \R^{n\times d_s}$, $H \in \R^{d_s \times n}$, $J \in \R^{d_s \times d_s}$ are constant matrices, and $J$ is positive semidefinite.

In the standard references on variational inequalities, the multivalued mapping $\cS(\cdot)$ is assumed to be stationary \cite{FaccPang03, PangStew08}.
To make connections with the standard formulation of evolution equations with time-varying domains \cite{AubiCell84, Mont93},  it is seen that \eqref{eq:sysDVIb}, \eqref{eq:sysDVIc} could be compactly, and equivalently, written as:
\begin{equation}\label{eq:incCone}
\eta(t) \in -\cN_{\cS(t)}(Hx(t)+J \eta(t)),
\end{equation}
where $\cN_{\cS(t)}(v(t))$ denotes the normal cone to the convex set $\cS(t)$ at $v(t)$, and is defined as:
\[
\cN_{\cS(t)}(v(t)) := \{\lambda \in \R^{d_s} \, |\, \langle \lambda, v'- v(t) \rangle \le 0, \forall \, v' \in \cS(t)\}.
\]
As convention, we let $\cN_{\cS(t)}(v(t)) := \emptyset$, for all $v(t) \not\in \cS(t)$.

One can also draw the analogies of the system class \eqref{eq:sysDVI} with the classical sweeping processes \cite{Moreau77}. When $J = 0$, one can interpret \eqref{eq:sysDVI} as follows: as long as $v(t) = Hx(t)$ is in the interior of the set $\cS(t)$, we get $\eta(t) = 0$ and \eqref{eq:sysDVI} reduces to an ODE $\dot x(t) = f(t,x(t))$ (for at least a small period of time) to satisfy the constraint $v(t) \in \cS(t)$, until $v(t)$ hits the boundary of the set $\cS(t)$.
At this moment, if the vector field $f(t,x(t))$ is pointed outside of the set $\cS(t)$, then any component of this vector field in the direction normal to $\cS(t)$ at $v(t)$ must be annihilated to maintain the motion of $v$ within the constraint set, see Figure~\ref{fig:setMapa}. In general, when $J \neq 0$, the vector $Hx(t)$ may not necessarily be contained in $\cS(t)$, and we seek a vector(s) $\eta(t)$ such that $v(t) = Hx(t)+J\eta(t) \in \cS(t)$ and $-\eta(t)$ is also normal to $\cS(t)$ at $v(t)$, see Figure~\ref{fig:setMapb}. In other words, the vector $\eta$ is defined implicitly and not explicitly for the general case $J \neq 0$.

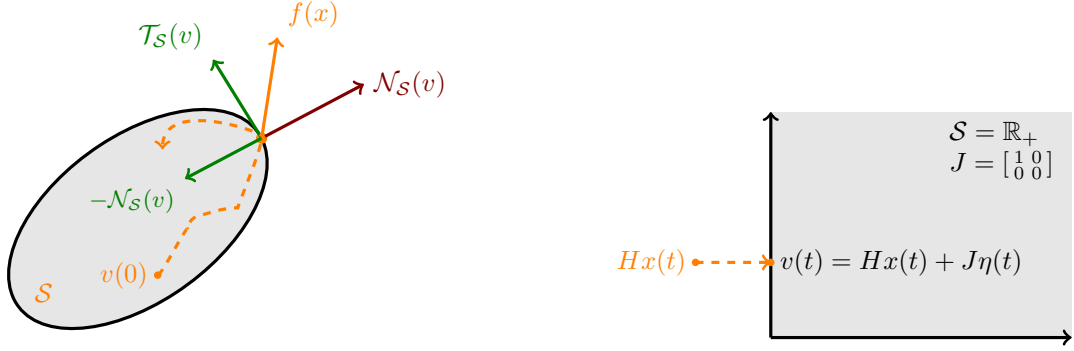
\begin{figure}
\centering
\begin{subfigure}[b]{0.49\linewidth}\centering
\begin{tikzpicture}[thick, xscale=0.75,yscale=0.85]
\def\myLength{72}
\def\myWidth{38}
\def\myAngle{30}
\def\horiz{0.867*\myLength}
\def\vert{0.5*\myLength}
\coordinate (start) at (10pt,-25pt);
\coordinate (origin) at (0,0);
\fill [{black!10}] (origin) ellipse [x radius = \myLength pt, y radius = \myWidth pt, rotate = \myAngle];
\draw [black, very thick] (origin) ellipse [x radius = \myLength pt, y radius = \myWidth pt, rotate = \myAngle];
\draw (-\myLength+15 pt,-\myWidth+5 pt) node[anchor=west, orange] {$\cS$};
\draw (start) node[orange, circle,fill,inner sep=1] {} node[anchor=east, orange] {$v(0)$};
\draw [orange, rounded corners, very thick, dashed,-] (start) .. controls +(right:2pt) and +(down:2pt) .. (25pt,0pt) -- (45pt,5pt) .. controls +(right:5pt) and +(down:5pt) .. (\horiz pt,\vert pt);
\draw (72*0.867pt,72*0.5pt) node[circle,fill,color=orange,inner sep=1.5] {} node[anchor=north] {};
\draw [Darkred, very thick, ->] (\horiz pt, \vert pt) -- +(25: 2cm) node[anchor=west] {$\cN_{\cS}(v)$};
\draw [orange, very thick, ->] (\horiz pt, \vert pt) -- +(80: 1.6cm) node[anchor=south west] {$f(x)$};
\draw [Darkgreen, very thick, ->] (\horiz pt, \vert pt) -- +(125: 1.5cm) node[anchor=south east] {$\cT_{\cS}(v)$};
\draw [Darkgreen, very thick, ->] (\horiz pt, \vert pt) -- +(205: 1.5cm) node[anchor=north east] {\small $-\cN_{\cS}(v)$};
\draw [orange, rounded corners, very thick, dashed,->] (\horiz pt,\vert pt) .. controls +(right:10pt) and +(up:25pt) .. +(-50 pt,-5pt);
\end{tikzpicture}
\caption{Analogy of DVIs with sweeping processes when $J=0$.}
\label{fig:setMapa}
\end{subfigure}
\begin{subfigure}[b]{0.49\linewidth}\centering
\begin{tikzpicture}
\fill [{black!10}] (0,0)--(0,3)--(4,3)--(4,0)--(0,0);
\draw [very thick,->] (0,0) -- (0,3);
\draw [very thick,->] (0,0) -- (4,0);
\draw (4,3) node[anchor=north east, text width=1.5cm,align=left] {$\cS = \R_+$ \newline $J = \left[ \begin{smallmatrix}1 & 0 \\ 0 & 0 \end{smallmatrix}\right]$};
\draw (-1,1) node[orange, circle,fill,inner sep=1] {} node[anchor=east, orange] {$Hx(t)$};
\draw [very thick, orange, dashed, ->] (-1,1) -- (0,1);
\draw (0,1) node[orange, circle,fill,inner sep=1] {} node[anchor=west] {$v(t) = Hx(t) + J \eta(t)$};
\end{tikzpicture}
\caption{The vector $J\eta(t)$ pushes $Hx(t)$ in the set $\cS = \R_+$.}
\label{fig:setMapb}
\end{subfigure}
\caption{Graphical depiction of set-valued components in system~\eqref{eq:sysDVI}.}
\label{fig:setMap}
\end{figure}

In what follows, we will use the standard notation, $\cl$, $\inn$, and $\rint$ to denote the closure,  interior, and the relative interior of a set respectively. The domain, range, and kernel of an operator are denoted by $\dom$, $\rge$, and $\ker$ respectively.

\subsection{Problem Formulation}\label{sec:prob}
As stated in the introduction, we basically consider two problems related to system class \eqref{eq:sysDVI}.

\subsubsection{Well-posedness of EVI~\eqref{eq:sysDVI}}\label{sec:probSol}
First, we are interested in knowing under what conditions on the system dynamics, a unique solution exists.
To formalize the solution concept for system class \eqref{eq:sysDVI}, we introduce the following set:
\begin{equation}\label{eq:defSetAdm}
\cS_{\text{adm}}^0 := \left\{ \overline x \in \R^n \, \left\vert \, 
\begin{aligned}
& \exists \, \eta \in \R^{d_s} \text{ satisfying } \\
& \eta \in -\cN_{\cS(0)}(H\overline x+J\eta)
\end{aligned}\right.\right\}.
\end{equation}
The following notion of solution for time-varying differential inclusions is borrowed from \cite[Definition~3.1]{Brez73}. 
\begin{defn}\label{def:sol}
The function $x:[0,T] \rightarrow \R^n$ is called a {\em strong} solution to system~\eqref{eq:sysDVI} with initial condition $x(0) \in \cS_{\text{adm}}^0$, if it is locally absolutely continuous, and satisfies $\eqref{eq:sysDVIa}$ for almost every $t \ge 0$, and \eqref{eq:sysDVIb}, \eqref{eq:sysDVIc} hold for each $t \ge 0$.
The continuous function $x:[0,T] \rightarrow \R^n$ is called a {\em weak} solution to system~\eqref{eq:sysDVI} with initial condition $x(0) \in \cS_{\text{adm}}^0$, if there exists a sequence of strong solutions $\{x_k\}_{k=1}^\infty$ to system~\eqref{eq:sysDVI} with initial condition $x_k(0) \in \cS_{\text{adm}}^0$ such that $\{x_k\}_{k=1}^\infty$ converges to $x$ uniformly.
\end{defn}

The constraint on the initial condition $x(0)$ is introduced so that there are no discontinuities in the solution $x$ at time $t =0$, and that the solution starts from the admissible domain. In Section~\ref{sec:sol}, we propose conditions on the system dynamics which guarantee well-posedness for the system class~\eqref{eq:sysDVI}.

\subsubsection{Output Regulation}
For this problem, we restrict ourselves to the case of linear vector fields, and the system class is defined as follows:
\begin{subequations}\label{eq:plant}
\begin{gather}
\dot x(t) = A x(t) + B u(t) + F x_r(t) + G \eta(t) \label{eq:planta} \\
v(t) = Hx(t) + J \eta(t) \label{eq:plantb} \\
\eta(t) \in -\cN_{\cS(t)}(v(t)) \label{eq:plantc}
\end{gather}
\end{subequations}
where $u$ is a control input, and $x_r:[0,\infty) \rightarrow \R^{d_r}$ is the reference signal that is generated from the following equations:
\begin{subequations}\label{eq:sysRef}
\begin{gather}
\dot x_r(t) = A_r x_r(t) + G_r \eta_r(t) \label{eq:sysRefa}\\
v_r(t) = H_rx_r(t) + J_r\eta_r(t) \label{eq:sysRefb} \\
\eta_r(t)  \in - \cN_{\cS(t)}(v_r(t)). \label{eq:sysRefc}
\end{gather}
\end{subequations}
The output regulation variable $w(\cdot)$ is defined as:
\begin{equation}\label{eq:refSig}
w(t) = C x(t) - C_r x_r(t).
\end{equation}
It will be assumed throughout the paper that system~\eqref{eq:sysRef} admits a solution (not necessarily unique) in the sense of Definition~\ref{def:sol}.
We say that the output regulation is achieved if there exists a control input $u$ such that the following properties are satisfied:
\begin{itemize}
\item {\bf Well-posedness:} For each initial condition in $\cS_{\text{adm}}^0$, there exists a unique solution to \eqref{eq:sysDVI} in the sense of Definition~\ref{def:sol}.
\item {\bf Regulation:} It holds that $\lim_{t\rightarrow \infty} w(t) = 0$.
\item {\bf Closed-loop stability:} The plant and controller dynamics have a globally asymptotically stable equilibrium at the origin  when $x_r\equiv 0$.
\end{itemize}

\subsection{Motivation}

The solution theory for system~\eqref{eq:sysDVI} could be useful in many aspects since several electrical and mechanical systems are modeled using this framework \cite{Brog16}.
To motivate the output regulation, we mention two possible applications of the proposed problem.

\subsubsection{Viability Control}
As first application, we consider the problem of finding a control input which maintains predefined constraints on the state trajectories of a dynamical system while achieving output regulation. This also relates to the problem of finding control inputs which make the set $\cS(t)$, for each $t \ge 0$, viable in the sense of Aubin \cite{AubiBaye11}. The standard formulations in viability theory treat questions related to existence of inputs which render a stationary set viable. In our work, we are focused on computing the control inputs, by formulating the input as the solution to an optimization problem, which guarantee evolution of the system trajectories within some time-varying sets.
Stated more precisely, suppose that we are given a plant described as:
\begin{equation}\label{eq:refViab}
\begin{aligned}
\dot x(t) &= A x(t) + Bu(t)+Fx_r
\end{aligned}
\end{equation}
and we would like to find a control $u(\cdot)$ which not only tracks a reference trajectory generated by the exosystem of form \eqref{eq:sysRef}, but also results in the state satisfying the constraint that $Hx(t) \in \cS(t)$, for all $t \ge 0$, where $\cS(\cdot)$ is some predefined closed and convex set-valued map.
This could be achieved by decomposing $u$ as $u:=u_\reg + u_\eta$, where we choose $u_\eta(t)$ as the solution of the following variational inequality:
\begin{equation}
\langle u_\eta(t) , v'-Hx(t) \rangle \ge 0, \quad \forall \, v' \in \cS(t).
\end{equation}
This choice of control input transforms the plant equation \eqref{eq:refViab} as follows:
\[
\dot x(t) = Ax(t) + Bu_\reg(t) + Bu_\eta(t) + F x_r(t)
\]
with the constraints
\begin{gather}
v(t) = Hx(t) \in \cS(t) \notag\\
\langle u_\eta(t), v'-Hx(t) \rangle \ge 0, \quad \forall \, v' \in \cS(t). \label{eq:uViab}
\end{gather}
The case of $\cS(\cdot)$ being a time-varying polyhedron was also considered as a special case in our previous work \cite{TanwBrog14b}.
In that case, \eqref{eq:uViab} is formulated as a linear complementarity problem which could be solved very efficiently using standard software packages.

\subsubsection{Regulation in Power Converters}
A large number of electrical circuits with nonsmooth devices (diodes, switches, etc.), such as power converters, are modeled using cone complementarity relations which is a special kind of variational inequality. To see this, we consider a closed convex polyhedral cone\footnote{We call $\cK \subset \R^{d_s}$ a closed convex cone, if for each $v_1,v_2 \in \cK$, and each $\alpha_1,\alpha_2 \ge 0$, we have $\alpha_1 v_1 + \alpha_2v_2 \in \cK$. The set $\cK$ is called a polyhedral set if there is a matrix $R$ such that $\cK=\{v \, \vert \, Rv \ge 0\}$.} $\cK\subseteq\R^{d_s}$.
Let $\cK^*$ denote the dual cone to $\cK$, defined as:
\begin{equation}\label{eq:defDual}
\cK^*:= \{\eta \in \R^{d_s} \,|\, \langle \eta, v \rangle \ge 0, \forall v \in \cK \}.
\end{equation}
Our framework allows us to consider the models of electrical systems of the following form:
\begin{subequations}\label{eq:sysElec}
\begin{gather}
\dot x(t) = A x(t) + Bu(t) + F x_r(t) + B_{\text{ext}} f_{\text{ext}} (t)+ G \eta(t) \\
v(t) = Hx(t) + J \eta(t) + h(t) \label{eq:sysElecb}\\
\cK \ni v(t) \perp \eta(t) \in \cK^*, \label{eq:sysElecc}
\end{gather}
\end{subequations}
where $f_{\text{ext}}$ and $h$ are sufficiently regular functions of time, and the goal is to design $u$ to solve an appropriate regulation problem.
The notation $a \perp b$ is a short-hand for writing $a^\top b = 0$. See \cite{AcarBonn11, VascCaml09} for examples.
Since $\cK$ is assumed to be a cone, the cone complementarity problem \eqref{eq:sysElecc} is equivalent to~\cite[Proposition~1.1.3]{FaccPang03}:
\[
\eta(t) \in -\cN_{\cK}(v(t)).
\]
If we let $\cS(t):=\cK - h(t) := \{\tilde v \in \R^{d_s}\, \vert \, \tilde v + h(t) \in \cK\}$, then the above inclusion is equivalently written as:
\[
\eta(t) \in -\cN_{\cS(t)}(Hx(t) + J\eta(t))
\]
and hence system~\eqref{eq:sysElec} can always be written in the form \eqref{eq:sysDVI}. The results on regulation of system class~\eqref{eq:plant} can thus be applied to some extent to solve regulation problems in circuits modeled by \eqref{eq:sysElec}.
\begin{rem}
In \eqref{eq:sysElec}, the variable $v$ has an additional time-dependent term which we didn't add in \eqref{eq:sysDVI}. A simple algebraic manipulation allows us to (re)define the set-valued map that brings us back to the case stated in \eqref{eq:sysDVI}. However, we limit ourselves to polyhedral sets when introducing this manipulation because for existence of solutions, the distance between the values of $\cS$ at any two time instants must be bounded by the difference of an absolutely continuous function at those time instants. The function $h$ provides this bound when $\cS$ is a polyhedral set-valued mapping (see Lemma~\ref{lem:hausBndK} in Appendix~\ref{app:litLemmas}), but in general, it may not.
\end{rem}

\subsection{Basics from Set-Valued Analysis}
In this section, we recall some basic results from convex analysis which are used in the subsequent sections for deriving the main results of this paper.
One can consult standard references, such as \cite{RockWets98}, for the results given here.

A set-valued map $\Phi(\cdot)$ is called {\em maximal monotone} if for each $x_1,x_2 \in \dom (\Phi)$, and $y_i \in \Phi (x_i)$, $i = 1,2$, we have $\langle y_2 - y_1, x_2 - x_1\rangle \ge 0$, and the graph of $\Phi$ cannot be extended any further while satisfying the monotonicity property.
An important thing to note is that there is a vast literature on the solution theory of differential inclusions where the multivalued operator on the right-hand side is maximal monotone \cite{Brez73}.

The following result allows us to draw connection between system~\eqref{eq:sysDVIa}, \eqref{eq:incCone} and the theory of maximal monotone operators. To do so, for a extended real-valued convex function $\psi:\R^{d_s} \to [-\infty, +\infty]$, we define the subdifferential of $\psi$ at $v \in \R^{d_s}$, denoted by $\partial \psi (v)$, as follows:
\[
\partial \psi(v) := \{\lambda \in \R^{d_s} \, |\, \langle \lambda, v'- v \rangle \le \psi(v') - \psi(v), \forall \, v' \in \dom f\}.
\]
\begin{prop}\label{prop:coneBasic}
Consider a nonempty, closed and convex set $\cS$ and let $\psi_\cS(\cdot)$ denote its indicator function, that is, $\psi_{\cS}(v) = 0$, if $v \in \cS$ and $\psi_S(v) = +\infty$ otherwise; Then
\begin{enumerate}
\item it holds that $\partial{\psi_\cS}(v) = \cN_{\cS}(v)$,
\item and $\cN_{\cS}(\cdot)$ is a maximal monotone operator.
\end{enumerate}
\end{prop}

In our approach, we would like to express \eqref{eq:sysDVI} as a differential inclusion by replacing $\eta$ with a set-valued operator.
In order to do that, one can see from $\eqref{eq:incCone}$ that we would need to define the ``inverse'' of the normal cone operator.
The theory of conjugate functions (or Legendre-Fenchel transforms) \cite[Chapter~11]{RockWets98} allows us to make this  connection.

\begin{defn}
For a function $\varphi:\R^n \rightarrow [-\infty,+\infty]$, the function $\varphi^*: \R^n \rightarrow [-\infty,+\infty]$ defined as:
\[
\varphi^*(\eta):= \sup_{v} \, \{\langle\eta,v\rangle - \varphi(v)\}
\]
is called the conjugate of $\varphi$.
For a closed convex set $\cS$, the conjugate of the indicator function $\psi_\cS(\cdot)$ is the support function $\sigma_\cS(\cdot)$ defined as:
\[
\sigma_\cS(\eta) = \sup_{v\in \cS} \, \langle v, \eta \rangle .
\]
\end{defn}
We now recall the following fundamental result:
\begin{prop}[{\cite[Proposition~11.3]{RockWets98}}]\label{prop:invConj}
For any proper, lower semicontinuous, convex function $\varphi(\cdot)$, one has $\partial \varphi^* = (\partial \varphi)^{-1}$ and $\partial \varphi = (\partial \varphi^*)^{-1}$. That is,
\begin{equation}
\eta \in \partial \varphi (v) \Longleftrightarrow v \in \partial \varphi^*(\eta).
\end{equation}
In particular, for a closed and convex set $\cS$:
\[
\eta \in \cN_{\cS}(v) \Longleftrightarrow v \in \partial \sigma_{\cS}(\eta).
\]
\end{prop}
Finally, the last notion we need is to quantify the distance between two sets in an appropriate manner.
\begin{defn}\label{def:haus}
The Hausdorff distance between two sets $S_1,S_2 \subseteq \R^{d_s}$, denoted by $d_{\Haus}(S_1,S_2)$, is defined as:
\[
d_{\Haus}(S_1,S_2) := \sup\left\{\sup_{v_1\in S_1}d(v_1,S_2),\sup_{v_2\in S_2}d(v_2,S_1)\right\}
\]
where $d(v,S) = \inf_{w\in \cS}\vert v-w \vert$ is the usual Euclidean distance between a point and the set.

\end{defn}

\section{Well-posedness of the Time-Varying EVIs}\label{sec:sol}
There is a considerable amount of literature on the solution theory for differential inclusions, depending on the structure of the set-valued map on the right-hand side.
One particular class of set-valued maps, which are interesting from the point of analysis and applications, are maximal monotone operators, and the solution theory for differential inclusions with such operators has been well-studied since the work of Br\'ezis \cite{Brez73}.
For our setup, such inclusions are of particular interest because we will now show that, when the set $\cS(t)$ is closed and convex valued, for each $t \ge 0$, then equation~\eqref{eq:sysDVI} can be equivalently written as a differential inclusion with time-varying maximal monotone operator plus a globally Lipschitz vector field on the right-hand side.

To see this, we use Propositions~\ref{prop:coneBasic} and \ref{prop:invConj}, and describe the relations in \eqref{eq:sysDVIa}, \eqref{eq:sysDVIb}, using a set-valued map for $\eta(t)$ as follows:
\begin{subequations}\label{eq:eqSetInc}
\begin{align}
& \langle v'- v(t), \eta(t)\rangle  \ge 0, \quad \forall \, v'\in \cS(t)\label{eq:eqSetInca}\\
\Longleftrightarrow \quad & \eta(t) \in -\partial \psi_{S(t)} (Hx(t)+J\eta(t))\label{eq:eqSetIncb}\\
\Longleftrightarrow  \quad & Hx(t) + J\eta(t) \in \partial \sigma_{\cS(t)}(-\eta(t))\label{eq:eqSetIncc}\\
\Longleftrightarrow  \quad & Hx(t) \in \left(\partial \sigma_{\cS(t)} + J \right) (-\eta(t)) \label{eq:eqSetIncd}\\
\Longleftrightarrow  \quad & -\eta(t) \in \left(\partial \sigma_{\cS(t)} + J \right)^{-1} (Hx(t)). \label{eq:eqSetInce}
\end{align}
\end{subequations}
Thus, if we introduce the operator $\Phi$ as follows:
\begin{equation}\label{eq:defPhi}
\begin{aligned}
\Phi: [0,\infty) \times \R^{d_s} & \rightrightarrows \R^{d_s} \\
(t,v) & \mapsto (\partial \sigma_{\cS(t)} + J)^{-1}(v),
\end{aligned}
\end{equation}
then system~\eqref{eq:sysDVI} can be equivalently written as the following differential inclusion:
\begin{equation}\label{eq:sysDI}
\dot x(t) \in f(t,x) - G \, \Phi(t,Hx(t)).
\end{equation}

It is an easy exercise to show that the operator $\Phi(t,\cdot)$ is maximal monotone for each $t \ge 0$ (see also the proof of Lemma~\ref{lem:maxPhi} in Section~\ref{sec:pfSolThm}) but it is not true in general that $G \, \Phi(t,H\cdot)$ is also maximal monotone.
If it is assumed that $f(t,x) = Ax + u(t)$, and that the LTI system defined using the matrices $(A,G,J,H)$ is passive and $\Phi$ is time-independent, then the maximal monotonicity of the multivalued operator on the right-hand side of \eqref{eq:sysDI} was proven in \cite{CamlSchu15}.
Our goal in this section is to generalize this result for the class of systems \eqref{eq:sysDI}, and the contribution of what follows in this section could be seen in following two regards:
\begin{itemize}
\item A direct approach (different than \cite{CamlSchu15}) to transform the right-hand side of \eqref{eq:sysDI} into a maximal monotone operator (with a minus sign) and a Lipschitz vector field, and then study the solutions of the resulting differential inclusion.
\item Generalize the system class by addressing time-dependent set-valued maps and nonlinear vector fields.
\end{itemize}

The main result highlighting these contributions is stated as follows:

\begin{thm} \label{thm:solMain}Assume that the following holds:
\begin{ass}[leftmargin=3em]

\item \label{ass:kerJ}The matrix $J$ is positive semidefinite and there exists a symmetric positive definite matrix $P$ such that $\ker(J+J^\top) \subseteq \ker (PG - H^\top)$.

\item \label{ass:regF}There exists a nonnegative locally essentially bounded function $\rho:[0,\infty) \rightarrow [0,\infty)$ such that
\[
|f(t,x_1) - f(t,x_2)| \le \rho(t) |x_1-x_2|, \quad \forall \, x_1,x_2 \in \R^n.
\]

\item \label{ass:qual} For each $t \ge 0$, $\rge H \, \cap \rint (\rge(\partial \sigma_{\cS(t)}+ J)) \neq \emptyset$.

\item \label{ass:minNorm} For every $t \ge 0$, and each $v \in \rge H \, \cap \rge(\partial \sigma_{\cS(t)}+ J)$, it holds that $\rge(J+J^\top) \cap (\partial\sigma_{S(t)} + J)^{-1}(v) \neq \emptyset$.

\item \label{ass:contS} It holds that $\rge J \subseteq \rge H$ and $\cS:[0,\infty) \rightrightarrows \R^{d_s}$ is closed and convex valued for each $t \ge 0$. Also, the mapping $\cS \cap \rge H$ varies in an absolutely continuous manner with time, that is, there exists a locally absolutely continuous function $\mu: [0,\infty) \rightarrow \R_+$, such that
\[
d_{\Haus}(\cS(t_1) \cap \rge H, \cS(t_2) \cap \rge H) \le |\mu(t_1) - \mu(t_2)|, \  \forall \, t_1, t_2 \ge 0,
\]
where $d_{\Haus}$ denotes the Hausdorff distance introduced in Definition~\ref{def:haus}.
\end{ass}

Then, for each $T \in [0,\infty)$, and $x(0)$ satisfying $Hx(0) \in \rge (\partial\sigma_{\cS(0)}+ \rge J)$, there exists a unique weak solution to \eqref{eq:sysDI}, and hence to \eqref{eq:sysDVI}, in the sense of Definition~\ref{def:sol} over the compact interval $[0,T]$.
\end{thm}

\subsection{Discussions}
\begin{enumerate}
\item If $J = 0$, then $\Phi(t,Hx) = \cN_{\cS(t)}(Hx)$, and \ref{ass:kerJ} basically implies that $PG = H^\top$. Well-posedness for that case was studied in \cite{Brog04, BrogGoel13}.
It is also seen that if the matrices $A,G,H,J$ satisfy a certain linear matrix inequality, then the assumption \ref{ass:kerJ} automatically holds \cite{CamlIann14} and this fact was used in \cite{CamlSchu15} for the static case where $\cS$ is not time-varying.

\item The global Lipschitz-like condition on the function $f(t, \cdot)$ is imposed because we are seeking solutions for all times, and not just the existence of an interval over which the solution is defined.

\item The constraint qualification \ref{ass:qual} is somewhat a standard assumption when dealing with problems on variational inequalities, as it ensures that the relative interior of the domain of the set-valued operator in \eqref{eq:sysDI} is nonempty. For the case $J=0$, we recall that $\rge \partial\sigma_{\cS(t)} = \cS(t)$ for each $t \ge 0$, and hence \ref{ass:qual} is consistent with \cite{Brog04}.

\item The condition \ref{ass:minNorm} provides an upper bound on the minimum norm element of the set $(\partial \sigma_{\cS(t)}+J)^{-1}(z)$ whenever $z$ is in the admissible domain (see Lemma~\ref{lem:minNormLip} in Section~\ref{sec:pfSolThm} and the discussion in Section~\ref{sec:commPf}). Our proof of Theorem~\ref{thm:solMain} requires a result on existence of solution of differential inclusions with time-varying maximal monotone operators \cite[Theorem~3]{KunzMont97}, which in turn requires the bound obtained by \ref{ass:minNorm}. It holds trivially when $J=0$.

\item The restrictive element of \ref{ass:contS} is $\rge J \subseteq \rge H$ (which would hold when $\rank H = d_s$, or $J = 0$). This condition combined with \ref{ass:qual} ensures that $\rge H \cap \cS(t) \neq \emptyset$, for each $t \ge 0$. The bound on the variation of $\cS \cap \rge H$ in \ref{ass:contS} is introduced to obtain absolutely continuous trajectories.
One can also think of relaxing this assumption to allow $\mu$ to be a function of bounded variation and this is briefly discussed in Section~\ref{sec:solBV}.
\end{enumerate}

\subsection{Proof of Theorem~\ref{thm:solMain}}\label{sec:pfSolThm}
The proof is based on several intermediate results which will be introduced as lemmas.
We first state some desired properties of the operator $\Phi$ introduced in \eqref{eq:defPhi} in Lemmas~\ref{lem:maxPhi}--\ref{lem:kerJ} .

\begin{lem}\label{lem:maxPhi}
For each $t \ge 0$, the operator $v \mapsto (\partial \sigma_{\cS(t)} + J)^{-1}(v)$ is maximal monotone.
\end{lem}

\begin{proof}
First, it is noted that $\sigma_{\cS(t)}$ is sublinear~\cite[Theorem~8.24]{RockWets98}, and hence a convex functional, so that $\partial \sigma_{\cS(t)}$ is maximal monotone for each $t \ge 0$, and $\rint (\dom \partial \sigma_{\cS(t)}) \neq \emptyset$ because $\cS(t)$ is convex-valued and $\rint(\cS(t)) \neq \emptyset$.
The matrix $J$ defines a monotone mapping, which is continuous, and thus maximal monotone \cite[Example~12.7]{RockWets98}.
Also, $\dom J$ is $\R^{d_s}$.
It now holds that $(\partial \sigma_{\cS(t)} + J)$ is maximal monotone because $\rint(\dom \partial \sigma_{\cS(t)}) \cap \rint (\dom J) \neq \emptyset$ \cite[Corollary~12.44]{RockWets98}.
Hence, $(\partial \sigma_{\cS(t)}+ J)^{-1}$ is also maximal monotone, because the inverse operation preserves maximal monotonicity~\cite[Exercise~12.8(a)]{RockWets98}.
\end{proof}

\begin{lem}\label{lem:kerJ}
For a given $t \ge 0$, and $x \in \R^n$, consider $\lambda_{1}, \lambda_{2} \in \Phi(t,Hx)$, then $\lambda_1 - \lambda_2 \in \ker (J+J^\top)$.
\end{lem}
\begin{proof}
For $i =1,2$, consider $\lambda_i \in (\partial \sigma_{\cS(t)} + J)^{-1}(Hx)$, then $Hx - J\lambda_i \in \partial \sigma_{\cS(t)} (\lambda_i)$. Since $\partial \sigma_{\cS(t)}$ is monotone, we have
\[
\langle Hx - J \lambda_1 - Hx + J \lambda_2, \lambda_1 - \lambda_2 \rangle \ge 0
\]
or
\[
\langle J (\lambda_1 - \lambda_2), \lambda_1 - \lambda_2 \rangle \le 0.
\]
It follows that $\lambda_1-\lambda_2 \in \ker (J+J^\top)$, because $J$ is positive semidefinite.
\end{proof}
For each $\lambda_\alpha \in \Phi(t,Hx)$ (with $t$ and $x$ fixed), it follows from Lemma~\ref{lem:kerJ} that the projection of $\lambda_{\alpha}$ on $\rge(J+J^\top)$, denoted by $\cP_J(\lambda_\alpha)$, is the same for any $\alpha$. We use the notation $\lambda^{\im} := \cP_J(\lambda_\alpha)$, so that $\lambda_\alpha$ can be written as
\[
\lambda_\alpha := \lambda^{\im} + \lambda_\alpha^{\ker}
\]
for some $\lambda^{\ker}_\alpha \in \ker (J+J^\top)$.

\begin{lem}\label{lem:minNormLip}
For each $t \ge 0$, and $x \in \dom \Phi(t,H \cdot)$, let $\lambda^{\im}(t,x):=\cP_J(\Phi(t,Hx))$. It holds that
\begin{equation}
\lambda^{\im}(t,x) = \argmin_{\lambda\in\Phi(t,Hx)} \ \  |\lambda|
\end{equation}
and the single-valued map $x \mapsto \lambda^{\im}(t,x)$ is Lipschitz continuous for each $t \ge 0$. That is, $\lambda^{\im}(t,x)$ is the least-norm element of the set $\Phi(t,Hx)$ and its dependence on $x$ is Lipschitz continuous. 
\end{lem}

\begin{proof}
We first show that, because of \ref{ass:kerJ} and \ref{ass:minNorm}, $\lambda^{\im}(t,x)$ belongs to the set $\Phi(t,Hx)$. To see this, consider $\lambda \in \Phi(t,Hx)$. Due to \ref{ass:minNorm}, there exists $\overline \lambda \in \rge(J+J^\top) \cap (\partial \sigma_{\cS(t)} + J)^{-1}(Hx)$. It then follows from Lemma~\ref{lem:kerJ} that
\begin{align*}
& \lambda - \overline \lambda \in \ker (J+J^\top) \\
\Rightarrow \quad & \lambda^{\im} + \lambda^{\ker} - \overline \lambda \in \ker (J+J^\top) \\
\Rightarrow \quad & (J+J^\top) \lambda^{\im} = (J+J^\top) \overline \lambda \\
\Rightarrow \quad & \lambda^{\im} - \overline \lambda \in \ker (J+J^\top) \\
\Rightarrow \quad & \lambda^{\im} = \overline \lambda
\end{align*}
and hence $\lambda^{\im}(t,x) \in \Phi(t,Hx)$.

Next, we recall that $\Phi(t,\cdot)$ being a maximal monotone operator, the set $\Phi(t,Hx)$ is closed and convex. Hence, it contains the least-norm element \cite[Section 3.12, Theorem~1]{Luen69} and let it be denoted by $\widetilde\lambda$. Assume ad absurdum that $\widetilde\lambda\neq\lambda^{\im} (t,x)$. Then, we must have, for each $\lambda \in \Phi(t,Hx)$, that
\[
\langle \widetilde \lambda, \lambda - \widetilde \lambda \rangle \ge 0.
\]
Due to Lemma~\ref{lem:kerJ}, we can assume that $\widetilde\lambda = \lambda^{\im} + \widetilde\lambda^{\ker}$, for some nonzero $\widetilde\lambda^{\ker} \in \ker (J+J^\top)$. Plugging $\lambda=\lambda^{\im}$ in the above inequality gives
\begin{align*}
& \langle \lambda^{\im} + \widetilde \lambda^{\ker}, -\widetilde \lambda^{\ker} \rangle \ge 0 \\
\Rightarrow \quad & - \langle \widetilde \lambda^{\ker}, \widetilde \lambda^{\ker} \rangle \ge 0
\end{align*}
but this is only possible with $\widetilde \lambda^{\ker} = 0$, a contradiction. Hence, $\widetilde\lambda=\lambda^{\im}$.

Finally, to show the Lipschitz continuity of the map $x \mapsto \lambda^{\im}(t,x)$, we consider two vectors $x_a, x_b$ and let $\lambda_i^{\im} = \mathcal{P}_J( (\partial \sigma_{\cS(t)}+J)^{-1}(Hx_i))$, for $i = a,b$. Then,
\[
Hx_i - J \lambda_i^{\im} \in \partial \sigma_{\cS(t)} (\lambda_i^{\im})
\]
and due to monotonicity of $\partial \sigma_{\cS(t)}$ for each $t \ge 0$, we have
\[
\langle Hx_a - J \lambda_a^{\im} - Hx_b + J \lambda_b^{\im}, \lambda_a^{\im} - \lambda_b^{\im} \rangle \ge 0
\]
or equivalently,
\begin{equation}\label{eq:monJH}
\langle J (\lambda_a^{\im} - \lambda_b^{\im}), \lambda_a^{\im} - \lambda_b^{\im} \rangle \le \langle H(x_a - x_b), \lambda_a^{\im} - \lambda_b^{\im} \rangle.
\end{equation}
Since $J$ is positive semidefinite, there exists $c_J > 0$, such that%
\footnote{Consider the matrix $V_J$ such that $J+J^\top = V_JV_J^\top$, and let $P_J := V_J(V_J^\top V_J)^{-1}V_J^\top$ be the orthogonal projection on $\rge(J+J^\top)$. In case, $(J+J^\top) = 0$, we let $P_J=0$. Then, $\vert \cP_J(\lambda) \vert ^2 = \left\langle P_J(\lambda), P_J(\lambda) \right\rangle = \lambda^\top V_J (V_J^\top V_J)^{-1} V_J^\top \lambda \le \|V_J^\top \lambda \|^2 \, \| (V_J^\top V_J)^{-1}\| = \| (V_J^\top V_J)^{-1}\| \left \langle V_J^\top \lambda, V_J^\top \lambda \right\rangle = \| (V_J^\top V_J)^{-1}\| \left \langle \lambda, V_JV_J^\top \lambda \right\rangle = 2 \| (V_J^\top V_J)^{-1}\| \left \langle J \lambda, \lambda \right\rangle$. To obtain \eqref{eq:posJnorm}, let $\lambda = \lambda_a^{\im} - \lambda_b^{\im}$, so that $\cP_J(\lambda) = \lambda$.
}
\begin{align}
|\lambda_a^{\im} - \lambda_b^{\im}|^2
& \le c_J \langle J (\lambda_a^{\im} - \lambda_b^{\im}), \lambda_a^{\im} - \lambda_b^{\im} \rangle \label{eq:posJnorm} \\
& \le c_J\langle H(x_a - x_b), \lambda_a^{\im} - \lambda_b^{\im} \rangle \notag\\
& \le c \, \vert x_a - x_b \vert \, \vert \lambda_a^{\im} - \lambda_b^{\im} \vert \notag
\end{align}
for some $c \ge 0$, which proves the Lipschitz continuity of the desired map.
\end{proof}

Since we have assumed that $\ker(J+J^\top) \subseteq \ker (PG-H^{\top})$, we must have $G \lambda_\alpha^{ker} = P^{-1}H^{\top} \lambda_{\alpha}^{\ker}$.
This allows us to rewrite \eqref{eq:sysDI} as follows:
\begin{equation}\label{eq:sysDecomp}
\begin{aligned}
\dot x(t) &= \! f(t,x) \!- G \lambda^{\im}(t) - P^{-1}H^\top \lambda_\alpha^{\ker}(t)\\
& = \! f(t,x) \! +\!(P^{-1}H^{\top}\!\!- G) \cP_J(\lambda_\alpha(t)) - P^{-1}H^\top \lambda_\alpha(t)\\
& = \! f(t,x) \! - P^{-1}(PG-H^\top) \cP_J(\lambda_\alpha(t)) - P^{-1}H^\top \lambda_\alpha(t) \\
\lambda_\alpha(t) & \in \Phi(t,Hx(t)).
\end{aligned}
\end{equation}
Let $R$ denote the square root of the matrix $P$ in \ref{ass:kerJ}, so that $R$ is also positive definite and symmetric.
Introduce the coordinate transformation $z = Rx$, then in the new coordinates, \eqref{eq:sysDecomp} is written as:
\begin{equation}\label{eq:newCoord}
\begin{aligned}
\dot z(t) &= R f(t,R^{-1}z) + (R^{-1}H^{\top}\!\! - RG) \cP_J(\lambda_\alpha(t)) - R^{-1}H^{\top} \lambda_{\alpha}(t) \\
\lambda_\alpha(t) &\in \Phi(t,HR^{-1}z(t)).
\end{aligned}
\end{equation}

\begin{lem}\label{lem:linMaxMon}
For each $t \ge 0$, the multivalued operator $R^{-1}H^\top \Phi(t,H R^{-1}\cdot)$ is maximal monotone.
\end{lem}
\begin{proof}
We basically use \cite[Theorem~12.43]{RockWets98} and \ref{ass:qual} to arrive at the result. We just need to show that $\Phi(t,\cdot)$ is maximal monotone for each $t \ge 0$, and that $\rge (HR^{-1}) \cap \rint(\dom (\Phi(t,\cdot))) \neq \emptyset$. The first condition holds due to Lemma~\ref{lem:maxPhi}. The latter holds due to the constraint qualification \ref{ass:qual} because $\dom (\Phi(t,\cdot)) = \rge (\partial \sigma_{\cS(t)} + J)$.
\end{proof}
As a result of Lemmas~\ref{lem:minNormLip}--\ref{lem:linMaxMon}, one can now write \eqref{eq:newCoord} as:
\begin{equation}\label{eq:sysLip}
\dot z(t) \in g(t,z) - \Psi (t,z)
\end{equation}
where
\begin{equation}\label{eq:defg}
g(t,z) := R f(t,R^{-1}z) - (R^{-1}H^{\top} - RG) \cP_J(\lambda_\alpha(t))
\end{equation}
is globally Lipschitz (in the second argument) and 
\begin{equation}\label{eq:defPsi}
\Psi (t,z) := R^{-1}H^{\top} \Phi(t,HR^{-1}z(t))
\end{equation}
is maximal monotone for each $t\ge 0$.

\begin{lem}\label{lem:solIncmu0}
Let $\tilde g:[0,T] \rightarrow \R^n$ be a locally absolutely continuous function. There exists a unique locally absolutely continuous function $\tilde z :[0,T] \rightarrow \R^n$ which is a strong solution to the differential inclusion
\begin{equation}\label{eq:sysFinalmu0}
\dot {\tilde z} (t) \in -\Psi(t,\tilde z) + \tilde g(t), \quad \tilde z(0) \in \dom \Psi(0,\cdot)
\end{equation}
in the sense of Definition~\ref{def:sol}.
\end{lem}

\begin{proof}
The proof is based on showing that the hypotheses of Theorem~\ref{thm:solMaxMonTV} given in Appendix~\ref{app:litLemmas} hold. We first show that $\Psi(\cdot,\cdot)$ satisfies the conditions {\em \ref{hyp:minNorm}} and {\em \ref{hyp:contDom}} listed in Theorem~\ref{thm:solMaxMonTV}, which is a result borrowed from \cite[Theorem~3]{KunzMont97}, and consequently \eqref{eq:sysFinalmu0} could be studied as a particular case.

For verifying {\em \ref{hyp:minNorm}}, it is observed that the least-norm element $\zeta^0(t,z)$ of $\Psi(t,\tilde z)$, for each $t\ge0$, $\tilde z \in \dom \Psi(t,\cdot)$, satisfies the bound
\begin{equation}\label{eq:minNormPsiBound}
\vert \zeta^0(t,\tilde z) \vert \le \| R^{-1}H^{\top} \| \, \vert \lambda^{\im}(t,R^{-1} \tilde z) \vert
\end{equation}
where $\lambda^{\im}(t,x) \in \Phi(t,Hx)$, so that $\lambda^{\im} (t,x) \in \cN_{\cS(t)}(Hx - J\lambda^{\im}(t,x))$. By using the definition of the normal cone, and the fact that $\lambda^{\im}(t,x) \in \rge (J+J^\top)$, we get
\begin{align*}
\vert \lambda^{\im}(t,x) \vert^2 & \le c_J \langle \lambda^{\im}(t,x), J \lambda^{\im}(t,x) \rangle \\
&\le c_J \langle \lambda^{\im}(t,x), Hx - v_t \rangle, \quad \forall \,v_t \in \cS(t),
\end{align*}
where $c_J$ is the constant introduced in \eqref{eq:posJnorm}. Using Cauchy-Schwarz inequality, we now get
\begin{equation}\label{eq:bndMinLambda}
\vert \lambda^{\im}(t,x) \vert \le c_J \, \vert Hx - v_t \vert, \quad \forall \,v_t \in \cS(t).
\end{equation}
Next, fix a point $w_0 \in \cS(0) \cap \rge H$, and choose $v_t^0 \in \cS(t) \cap \rge H$ such that
\[
\vert w_0 - v_t^0 \vert = d (w_0, \cS(t) \cap \rge H)
\]
which always exists because $\cS(t) \cap \rge H$, for each $t \ge 0$, is closed and convex.
Returning to \eqref{eq:bndMinLambda}, and letting $v_t = v_t^0$, we get
\[
\begin{aligned}
\vert \lambda^{\im}(t,x) \vert & \le c_J \|H\| \ \vert x \vert + c_J \vert v_t^0 - w_0 \vert + \vert w_0 \vert.
\end{aligned}
\]
Since $d(w_0, \cS(t)) \le d_{\Haus}(\cS(0)  \cap \rge H,\cS(t) \cap \rge H) \le \vert \mu(t) - \mu (0) \vert$, where the last inequality is due to \ref{ass:contS}, we obtain
\[
\vert \lambda^{\im}(t,x) \vert \le c_1 \, \vert x \vert + c_2 \mu(t) + c_3(\vert w_0\vert + \vert \mu_0 \vert)
\]
for some constants $c_1, c_2, c_3 \ge 0$.
By substituting this inequality in \eqref{eq:minNormPsiBound} with $x = R^{-1} \tilde z$, hypothesis \ref{hyp:minNorm} is seen to hold.

For the hypothesis {\em \ref{hyp:contDom}} (recalling again Theorem~\ref{thm:solMaxMonTV} given in Appendix~\ref{app:litLemmas}), we show that there is a locally absolutely continuous function $\tilde \mu$ such that, for some fixed $t_1, t_2 \ge 0$, the map $\Psi$ satisfies
\begin{equation}\label{eq:bndSetVar}
d_{\svm} (\Psi(t_1, \cdot), \Psi(t_2,\cdot)) \le \vert \tilde \mu(t_1) - \tilde \mu(t_2) \vert
\end{equation}
where $d_{\svm}(\Psi_1,\Psi_2)$ denotes a certain distance between set-valued mappings $\Psi_1$ and $\Psi_2$, see \eqref{eq:defDistSvm} below.
In the sequel, we show that $d_{\svm} (\Psi(t_1, \cdot), \Psi(t_2,\cdot)) \le d_{\Haus} (\cS(t_1) \cap \rge H, \cS(t_2) \cap \rge H)$ and by Assumption~\ref{ass:contS}, the bound \eqref{eq:bndSetVar} holds.
Towards this end, it is observed that
\begin{align}
& d_{\svm} (\Psi(t_1, \cdot), \Psi(t_2,\cdot)) \notag\\
& = \sup \left\{\frac{\langle H^\top (\lambda_1 - \lambda_2), z_2 - z_1 \rangle}{1 + |H^\top \lambda_1| + |H^\top \lambda_2|}, \lambda_i \in (\partial\sigma_{\cS(t_i)} + J)^{-1} (Hz_i), i=1,2 \right\} \label{eq:defDistSvm}\\
& = \sup_{\lambda_i \in \cN_{\cS(t_i)}(H z_i - J\lambda_i)}\left\{\frac{\langle \lambda_1 - \lambda_2,  H z_2 - J \lambda_2 - H z_1 + J \lambda_1\rangle - \langle \lambda_1 - \lambda_2, J(\lambda_1-\lambda_2)\rangle}{1 + |H^\top \lambda_1| + |H^\top \lambda_2|}\right\} \notag\\
& \le \sup_{\lambda_i \in \cN_{\cS(t_i)}(H z_i-J\lambda_i)}\left\{\frac{\langle \lambda_1 - \lambda_2, H z_2 - J \lambda_2 - H z_1 + J \lambda_1 \rangle}{1 + |H^\top \lambda_1| + | H^\top \lambda_2|}\right\} \notag
\end{align}
where the last inequality is due to the fact that $J$ is positive semidefinite. To compute an upper bound on the expression in the numerator,
let $v_i:=H z_i - J\lambda_i \in \cS(t_i) \cap \rge H$, which is always possible since $\rge J \subset \rge H$.
Let $w_i \in \cS(t_i) \cap \rge H$ be such that
\[
\vert w_1 - v_2 \vert = d (v_2, \cS(t_1) \cap \rge H), \quad \text{and} \quad \vert w_2 - v_1 \vert = d (v_1, \cS(t_2) \cap \rge H).
\]
Using the fact that $\langle \lambda_1, w_1 - v_1 \rangle \le 0$ because $\lambda_1 \in \cN_{\cS(t_1)}(v_1)$, we obtain
\[
\langle \lambda_1, v_2 - v_1 \rangle  = \langle \lambda_1, v_2 - w_1 + w_1 - v_1 \rangle
\le \langle \lambda_1, v_2 - w_1 \rangle.
\]
Let $\lambda_1^H := H \widehat \lambda_1$ denote the orthogonal projection of $\lambda_1$ on $\rge H$, so that $\widehat \lambda_1$ is any vector that satisfies $(H^\top H) \widehat \lambda_1 = H^\top \lambda_1$, and there exists a constant $c_H$ such that $\vert \lambda_1^H \vert \le c_H \vert H^\top \lambda_1 \vert$. Since $\lambda_1 - \lambda_1^H \in \rge(H)^\bot = \ker (H^\top)$, and $v_2, w_1 \in \rge H$, we get
\[
\begin{aligned}
\langle \lambda_1, v_2 - v_1 \rangle & \le \langle \lambda_1, v_2 - w_1 \rangle = \langle \lambda_1^H, v_2 - w_1 \rangle \\
& \le \vert \lambda_1^H \vert \cdot \vert v_2 - w_1 \vert \le c_H \vert H^\top \lambda_1 \vert d (v_2, \cS(t_1) \cap \rge H)\\
& \le c_H \vert H^\top \lambda_1 \vert d_{\Haus}(\cS(t_1) \cap \rge H ,\cS(t_2) \cap \rge H)
 \le c_H \vert H^\top \lambda_1 \vert \cdot \vert \mu (t_2) - \mu(t_1) \vert .
\end{aligned}
\]
Similarly, we can obtain
\[
\langle \lambda_2, v_1- v_2 \rangle \le c_H \vert H^\top \lambda_2 \vert \cdot \vert \mu (t_2) - \mu(t_1) \vert ,
\]
and hence
\[
\langle \lambda_1 - \lambda_2, v_2- v_1 \rangle \le c_H (\vert H^\top \lambda_1 \vert + \vert H^\top \lambda_2 \vert) \vert \mu (t_2) - \mu(t_1) \vert.
\]
It thus follows that
\begin{align*}
d_{\svm} (\Psi(t_1, \cdot), \Psi(t_2,\cdot)) & \le \sup_{\lambda_i \in \cN_{\cS(t_i)}(H z_i-J\lambda_i)}\left\{ \frac{c_H \vert \mu (t_2) - \mu(t_1) \vert (\vert H^\top \lambda_1\vert + \vert H^\top \lambda_2\vert)}{1 + |H^\top \lambda_1| + |H^\top \lambda_2|}\right\} \\
& \le c_H \vert \mu (t_2) - \mu(t_1) \vert.
\end{align*}
The result of Lemma~\ref{lem:solIncmu0} now follows by applying Theorem~\ref{thm:solMaxMonTV} from Appendix~\ref{app:litLemmas}.
\end{proof}

Next, we use the result of Lemma~\ref{lem:solIncmu0} to obtain a bound on $|z_1(t) - z_2(t)|$, for each $t \ge 0$, where $z_i$, $i = 1,2$, satisfies
\begin{equation}
\dot z_i (t) \in -\Psi(t,z_i) + \tilde g_i(t)
\end{equation}
for some locally absolutely continuous functions $\tilde g_i:[0,\infty) \rightarrow \R^n$.
To compute this bound, it is seen that the monotonicity of $\Psi(t,\cdot)$ yields
\begin{align*}
\frac{1}{2} \frac{d}{dt}|z_1(t)- z_2(t)|^2 &= \langle \dot z_1(t) - \dot z_2(t), z_1(t) - z_2(t) \rangle \\
 & \le \langle \tilde g_1(t) - \tilde g_2(t), z_1(t) - z_2(t)\rangle \\
 & \le |\tilde g_1(t) - \tilde g_2(t)| \cdot |z_1(t) - z_2(t)|,
\end{align*}
where the first inequality is due to the monotonicity of $\Psi(t,\cdot)$, and the second inequality is due to Cauchy-Schwarz.
Letting $Z(t):=|z_1(t)- z_2(t)|^2$, the above inequality is rewritten as:
$\dot Z(t) \le 2|\tilde g_1(t) - \tilde g_2(t)|\sqrt{Z(t)}$. Applying the comparison lemma for solution of ODEs \cite[Lemma~3.4]{Khalil02}, we get
\begin{equation}\label{eq:estGron}
|z_1(t) - z_2(t)| \le |z_1(0)- z_2(0)|+ \int_0^t |\tilde g_1(s) - \tilde g_2(s)| \, ds.
\end{equation}

Coming back to the question of existence of solutions to \eqref{eq:sysLip}, we make use of the estimate obtained in \eqref{eq:estGron}.
Towards this end, consider a sequence of solutions with $z_1(t) = z(0)$, $t \ge 0$, and $z_{i+1}(\cdot)$, for $i \ge 1$ obtained as a solution to the inclusion
\[
\dot z_{i+1}(t) \in - \Psi(t,z_{i+1}) + g(t, z_i (t))
\] 
with initial condition $z_i(0) = z(0)$, $i \in \N$.
Under the assumption \ref{ass:regF}, and recalling Lemma~\ref{lem:minNormLip} and \eqref{eq:defg}, it follows that for each compact interval $[0,T]$, there exists a constant $\overline \rho_T$ such that
\[
\vert g(t, z_i (t)) - g(t, z_{i-1}(t))\vert \le \overline \rho_T \vert z_i(s) - z_{i-1}(s) \vert.
\]
Then, from \eqref{eq:estGron}, we have, for each $t \in [0,T]$:
\[
|z_{i+1}(t) - z_{i}(t)| \le \int_0^t \overline \rho_T \vert z_i(s) - z_{i-1}(s) \vert \, ds
\]
which through induction leads to:
\[
|z_{i+1}(t) - z_{i}(t)| \le \frac{(\overline \rho_T \,  t)^i}{i!} \|z_{2} - z_1\|_{\cL_{\infty}([0,T],\R^n)}.
\]
Thus, the sequence $\{z_i\}_{i=1}^\infty$ converges uniformly on every compact interval $[0,T]$, and hence
\[
z(t) := \lim_{i\rightarrow \infty} z_i(t)
\]
is a weak solution to system~\eqref{eq:sysLip}.

Uniqueness can be shown to hold easily for the inclusion \eqref{eq:sysLip} due to monotonicity of the operator $\Psi(t,\cdot)$.
Since we have only introduced bijective operations in arriving from system~\eqref{eq:sysDVI} to \eqref{eq:sysLip}, the conclusion also holds for system~\eqref{eq:sysDVI}.
This concludes the proof of Theorem~\ref{thm:solMain}.

\subsection{Comments and Examples}\label{sec:commPf}

We provide some additional comments on the assumptions \ref{ass:qual} and \ref{ass:minNorm} used in the statement of Theorem~\ref{thm:solMain}.

\subsubsection{Alternative formulation of \ref{ass:qual}}
The assumption \ref{ass:qual} was primarily used to satisfy the condition that renders the map $x \mapsto H^\top \Phi(t,Hx)$ maximal monotone. It is natural to ask whether we can say something more about the set $\dom \Phi(t,\cdot) = \rge (\partial\sigma_{\cS(t)}+J)$. For a general characterization of this set, see \cite[Section~4]{BrogGoel13}.
Here, we want to point out a special case when the matrix $J$ is positive definite.

\begin{prop}\label{lem:domPhi}
If the matrix $J$ is positive definite, then for each $t \ge 0$, we have
\begin{gather}
\rint(\dom \Phi(t,\cdot)) = \rint(\rge(\partial \sigma_{\cS(t)}+J))= \rint(\cS(t)) + \rge J.
\end{gather}
\end{prop}
\begin{proof}
By definition, for each $t \ge 0$, $\dom \Phi(t,\cdot) = \rge (\partial \sigma_{\cS(t)}+ J)$, and it is observed that
\begin{align*}
\rint(\rge(\partial \sigma_{\cS(t)}+J)) &= \rint(\cl(\rge(\partial \sigma_{\cS(t)}+J))) \\
& = \rint (\cl (\rge \partial \sigma_{\cS(t)} )+ \cl(\rge J)) \\
& = \rint (\rge \partial \sigma_{\cS(t)} + \rge J) \\
& = \rint (\rge \partial \sigma_{\cS(t)}) + \rint(\rge J)\\
& = \rint (\cS(t)) + \rge J,
\end{align*}
where the first and third equality follow from \cite[Theorem~6.3]{Rock70}, and the fourth equality is due to \cite[Corollary~6.6.2]{Rock70} because the sets under consideration are convex. The second equality is due to Theorem~\ref{thm:sumMon} in Appendix~\ref{app:litLemmas} and holds because $\partial \sigma_{\cS(t)}$ and $J$ define monotone operators, and satisfy the two conditions required to invoke Theorem~\ref{thm:sumMon}.
\end{proof}

\subsubsection{Necessity of \ref{ass:minNorm}}
Lemma~\ref{lem:minNormLip} is a central ingredient of the proof of Theorem~\ref{thm:solMain}. It essentially relies on the assumption \ref{ass:minNorm}. The Lipschitz continuity of the least norm element of the set $\Phi(t,z)$ with respect to $z \in \dom \Phi(t,\cdot)$ does not always hold, see also Example~\ref{ex:nonLip} below. We need this regularity condition to invoke \cite[Theorem~3]{KunzMont97} in Lemma~\ref{lem:solIncmu0} for existence of solutions to differential inclusions with time-varying maximal monotone operators. Another result on the solutions of such systems, which does not require restrictions on least norm element, appears in \cite{Vlad91}. But \cite[Theorem~7.1]{Vlad91} is only valid for {\em regular} set-valued operators, which requires $\dom(\Phi(t,\cdot))$ to have nonempty {\em interior} for each $t \ge 0$ and for our purposes it is much more restrictive than \cite[Theorem~3]{KunzMont97}.
\begin{exam}\label{ex:nonLip}
Consider a system with $f(t,x) \equiv 0$, for each $t \ge 0$,
\[
\cS(t) = \cS = \{(x_1,x_2)\in \R^2 \, \vert \, x_2 \ge x_1^2\}, \quad \forall \, t \ge 0
\]
and the matrices are given by
\[
H = \begin{bmatrix} 1 & 0 \\ 0 & 1\end{bmatrix}, \quad G = \begin{bmatrix} 1 & 0 \\ 0 & 1\end{bmatrix}, \quad J = \begin{bmatrix} 0 & 0 \\ 0 & 1\end{bmatrix}.
\]
For this system, conditions \ref{ass:kerJ}, \ref{ass:regF}, \ref{ass:qual}, \ref{ass:contS} hold. The condition $\ref{ass:minNorm}$ does not hold and for this reason, the mapping $x \mapsto \lambda(x)$ given by $\lambda (x) = 0$, if $x \in \cS$
\[
\begin{pmatrix} x_1 \\ x_2 \end{pmatrix}
\mapsto
\begin{pmatrix} 2x_1^3 \\ -x_2-x_1^2 \end{pmatrix}, \quad (x_1,x_2) \not \in \cS
\]
is not Lipschitz. The example shows that we cannot have Lipschitz continuity of $\lambda^{\im}(x)$ if \ref{ass:minNorm} does not hold.
\end{exam}

\begin{exam}
Consider the same system as in Example~\ref{ex:nonLip}, but with
\[
\cS(t) = \{(x_1,x_2)\in \R^2 \, \vert \, x_2 \ge 0, x_1 \ge 0\}, \quad \forall \, t \ge 0.
\]
It is easily verified that all the assumptions of Theorem~\ref{thm:solMain} hold in this case. One could even replace $J$ with a nonsymmetric matrix given by $\left[\begin{smallmatrix} 0 & -1 \\ 1 & 1\end{smallmatrix}\right]$ to arrive at the same conclusion.
\end{exam}

%
\subsection{The Case of Bounded Variation}\label{sec:solBV}

Thus far, we have considered solutions of system \eqref{eq:sysDVI} which are locally absolutely continuous. This is primarily because the variation in the set-valued mapping $\cS$ is bounded by an absolutely continuous function.
If one relaxes $\cS$ to evolve such that the Hausdorff distance between the values of $\cS$ at any two time instants is bounded by the variation of a locally BV function,\footnote{
For an interval $I \subseteq \R$, and a function $f:I \rightarrow \R^n$,
the variation of $f(\cdot)$ over the interval $I$ is the supremum of $\sum_{i=1}^k |f(s_i) - f(s_{i-1})|$ over the set of all finite sets of points $s_0 < s_1 < \cdots < s_k$ (called partitions) of $I$.
When this supremum is finite, the mapping $f(\cdot)$ is said to be of {\em bounded variation} (BV) on $I$.
We say that $f(\cdot)$ is of {\em locally bounded variation} if it is of bounded variation on each compact subinterval of $I$. For a BV function $f$, it holds that the right and left limits of $f$ are defined everywhere, and we use the notation $f(t^+):=\lim_{s \searrow t} f(s)$ and $f(t^-):=\lim_{s \nearrow t} f(s)$.} then the resulting solution will be a locally BV function. We are thus interested in the scenario where the restriction on function $\mu$ in \ref{ass:contS} is relaxed as follows:
\begin{assBV}[leftmargin=5em, start=5]
\item\label{ass:contSbv}
there exists a right-continuous locally BV function $\mu_f : [0,\infty) \rightarrow [0,\infty)$, whose variation over the interval $(0,t]$ is denoted by $\mu(t) := var(\mu_{f_{(0,t_1]}})$, and it satisfies
\[
d_{\Haus}(\cS(t_1) \cap \rge H, \cS(t_2) \cap \rge H) \le d\mu ((t_1,t_2]) = \mu(t_2) - \mu(t_1), \  \forall \, 0 \le t_1 \le t_2,
\]
where $d\mu$ denotes the Lebesgue-Stieltjes measure associated with $\mu$.
\end{assBV}

When working with BV functions, the equation~\eqref{eq:sysDVI} has to be interpreted in the sense of measures associated with the BV functions, and care must be taken in defining the appropriate solution concept.

\begin{defn}\label{def:solBV}
Consider system \eqref{eq:sysDVI} under the assumption \ref{ass:contSbv}. A right-continuous locally BV function $x:[0,\infty) \rightarrow \R^n$ is called a {\em strong} solution to system~\eqref{eq:sysDVI} if there exists a Radon measure $d\nu$, which is absolutely continuously equivalent\footnote{
A measure $d\mu$ is absolutely continuous with respect to another measure $d\nu$ if $d\nu(I)=0$ implies $d\mu(I) = 0$. The measures $d\nu$ and $d\mu$ are absolutely equivalent if they are absolutely continuous with respect to each other. By the Radon-Nikodym theorem, the density of Lebesgue measure $dt$ relative to $d\nu$, defined as, 
$\frac{dt}{d\nu}(s) = \lim_{\varepsilon \rightarrow 0} \frac{dt((s-\varepsilon,s+\varepsilon))}{d\nu((s-\varepsilon,s+\varepsilon))}$ is then well-defined.}
to $dt + d\mu$ such that the differential measure $dx$ is absolutely continuous with respect to $d\nu$, $\frac {dx}{d\nu} \in \cL_1^{loc}(I, \R^n;\nu)$, and the following relation holds:
\begin{gather}\label{eq:defSolBV}
\frac{dx}{d\nu}(t) \in f(t,x) \frac{dt}{d\nu}(t) - G (\partial \sigma_{\cS(t)}+J)^{-1}(Hx(t^+)), \quad \nu - a.e. \ t \in [0,\infty).
\end{gather}
\end{defn}

The weak solutions can be defined similarly as in Definition~\ref{def:sol}.
This concept of solution is borrowed from the literature on sweeping processes with BV solutions, see for example~\cite[Section~2.2]{EdmoThib06} or \cite{AdlyHadd14}, and this solution concept is independent of the choice of the measure $\nu$ because $(\partial \sigma_{\cS(t)}+J)^{-1}(z)$ is a cone for any $z \in \dom (\partial \sigma_{\cS(t)}+J)^{-1}$.
An important observation that can be made from \eqref{eq:defSolBV} is that, at time $t_i$, if there is a jump in $\mu$, so that the measure $d\mu$ is supported on the singleton $\{t_i\}$, then a jump in the state trajectory $x$ at $t_i$ is also possible, and this jump is represented by:
\[
x(t_i^+) - x(t_i^-) \in - G (\partial \sigma_{\cS(t)}+J)^{-1}(Hx(t_i^+)).
\]
When $G,H$ are identity matrices, and $J = 0$, this corresponds to
\[
x(t_i^+) -x(t_i^-) \in -\cN_{\cS(t_i)}(x(t_i^+))
\]
which is equivalent to solving the following quadratic optimization problem over a convex set:
\[
x(t_i^+) = \argmin_{v\in\cS(t_i)} \|x(t_i^-) - v\|^2.
\]
The result of Theorem~\ref{thm:solMain} can now be generalized as follows:
\begin{cor}
Assume that system~\eqref{eq:sysDVI} satisfies \ref{ass:kerJ}, \ref{ass:regF}, \ref{ass:qual}, \ref{ass:minNorm} and \ref{ass:contSbv}; then there exists a unique {weak} solution $x$ in the sense of Definition~\ref{def:solBV}.
\end{cor}

The proof of this result follows the same blueprint as laid out in the proof of Theorem~\ref{thm:solMain}. The difference starts appearing when we invoke Lemma~\ref{lem:solIncmu0}, as the existence of solution must now be proven for \eqref{eq:sysFinalmu0} when the time-dependence on the right-hand side is only of bounded variation. The result from \cite[Theorem~1]{KunzMont97} can now be used, and the rest of the arguments can be tailored accordingly to arrive at the desired result.
For the sake of briefness, the details are omitted in this article.


\subsection{Cone Complementarity Systems}

The framework of EVIs is especially useful for studying models of electrical circuits with complementarity relations. The result of Theorem~\ref{thm:solMain} allows us to capture well-posedness for a certain class of complementarity systems.
In particular, we can treat the cone complementarity systems described by
\begin{subequations}\label{eq:sysForceDVI}
\begin{gather}
\dot x(t) = f(t,x) + G \eta(t) \label{eq:sysForceDVIa} \\
v(t) = Hx(t) + J \eta(t) + h(t), \label{eq:sysForceDVIb} \\
\cK \ni v(t) \perp \eta(t) \in \cK^* \label{eq:sysForceDVIc},
\end{gather}
\end{subequations}
where $h:[0,\infty) \rightarrow \R^{d_s}$ is a Lebesgue-measurable function, and $\cK$ is a closed, convex polyhedral cone, and $\cK^*$ denotes the dual cone as defined in \eqref{eq:defDual}. We introduce the set-valued mapping
\begin{equation}\label{eq:defSconv}
\cS(t) := \left\{z \in \R^{d_s} \, \vert \, z+h(t) \in \cK \right\}
\end{equation}
and reformulate the dynamics of \eqref{eq:sysForceDVI} in the form of \eqref{eq:sysRef}.
Towards this end, we recall the following fundamental relation from convex analysis~\cite[Proposition~1.1.3]{FaccPang03}:
\begin{equation}\label{eq:compConv}
\cK\ni v \perp \eta \in \cK^* \Leftrightarrow 
\eta \in -\partial \psi_{\cK}(v) = -\cN_{\cK}(v)
\end{equation}
where $\psi_S$ denotes the indicator function of the set $S$.
One can effectively rewrite \eqref{eq:sysForceDVIb}-\eqref{eq:sysForceDVIc} as
\begin{align}
\eta(t) \in -\cN_{\cS(t)} (Hx(t) + J\eta(t)).
\end{align}
The result of Theorem~\ref{thm:solMain} can now be invoked by checking the conditions \ref{ass:qual}, \ref{ass:minNorm}, \ref{ass:contS} for $\cS$. But, because of the added structure on the set $\cK$, and the regularity assumption on $h$, it is possible to rewrite the conditions for the well-posedness of \eqref{eq:sysForceDVI}.
\begin{cor}\label{cor:solLCCS}
Consider that system~\eqref{eq:sysForceDVI} satisfies the assumptions \ref{ass:kerJ}, \ref{ass:regF}, and suppose that
\begin{itemize}
\item The cone $\cK$ and the matrix $J$ satisfy the condition\footnote{The condition~\eqref{eq:condRintCone} is not very restrictive because we always have the inclusion $\cK-h(t) = \partial\sigma_{\cK-h(t)}(0) \subseteq \rge (\partial \sigma_{\cK-h(t)}+J)$. However, the set inclusions are not preserved by the $\rint$ operator in general.}
\begin{subequations}
\begin{equation}\label{eq:condRintCone}
\rint (\cK - h(t)) \subseteq \rint (\rge (\partial \sigma_{\cK-h(t)}+J)).
\end{equation}
\begin{equation}\label{eq:inclJH}
J \cK^* \subseteq \rge H.
\end{equation}
\end{subequations}
\item The matrix $H$ is such that
\footnote{Equation~\eqref{eq:consH} is to be interpreted in the sense that, for each $h \in \R^{d_s}$, there exists $x \in \R^n$, and $v \in \cK$, such that $h = Hx - v$.
Clearly, if $H$ defines a surjective mapping, then \eqref{eq:consH} holds but it may also hold in more general cases.
}
\begin{equation}\label{eq:consH}
\rge H - \cK = \R^{d_s}.
\end{equation}
\item For each $x \in \R^n$ and $t \ge 0$, if the set $\Lambda_x(t) := \{\eta \in \cK^* \, \vert \, v= Hx + J \eta + h(t) \in \cK, \langle\eta,v\rangle = 0\}$ has a nonzero element, then
\begin{equation}\label{eq:minNormCone}
\Lambda_x(t) \cap \rge (J+J^\top) \neq \emptyset.
\end{equation}
\end{itemize}
If $h$ is locally absolutely continuous (respectively, right-continuous BV), then there exists a unique {weak} solution to \eqref{eq:sysForceDVI} which is continuous (resp.~right-continuous BV).
\end{cor}

To prove this result as an application of Theorem~\ref{thm:solMain}, we basically show that the assumptions \ref{ass:qual}, \ref{ass:minNorm} and \ref{ass:contS} (resp.~\ref{ass:contSbv}) hold for system \eqref{eq:sysForceDVI}. It is first seen that the condition \eqref{eq:minNormCone} implies \ref{ass:minNorm}, because $\Lambda_x (t) = -(\partial \sigma_{\cS(t)}+J)^{-1}(Hx)$. This equivalence is obtained by showing that the two sets are contained in each other. Indeed, first take $\eta_x \in -(\partial \sigma_{\cS(t)}+J)^{-1}(Hx)$, then
\begin{align*}
& \eta_x \in -(\partial \sigma_{\cS(t)}+J)^{-1}(Hx) \\
\Longleftrightarrow \quad & \eta_x \in -\cN_{\cS(t)}(Hx+J\eta_x) & &\text{(from Proposition~\ref{prop:invConj})}\\
\Longleftrightarrow \quad & \eta_x \in -\cN_{\cK}(Hx+J\eta_x + h(t)) & &\text{(from \eqref{eq:defSconv})}\\
\Longleftrightarrow \quad & \eta_x \in \Lambda_x (t) & &\text{(from \eqref{eq:compConv})}.
\end{align*}
The inclusion in the other direction follows the same arguments.

For verifying \ref{ass:qual}, we first observe that \eqref{eq:consH} says: for each $z \in \R^{d_s}$, $\rge H \cap (\{z\}+\cK) \neq \emptyset$. For a given $t \ge 0$, choose $z = v_{\inn} - h(t)$, for some $v_{\inn}$ in the relative interior of $\cK$. Then, there exist $v \in \cK$ and $x \in \R^n$ such that $Hx = v + v_{\inn} - h(t)$. Since $\cK$ is a convex cone, it can be shown that $v + v_{\inn} \in \rint \cK$, and hence, for each $t \ge 0$, $\rge H$ intersects nontrivially with $\rint(\cK -h(t))\subset \rint(\rge(\partial\sigma_{\cK-h(t)}+J))$, where the last inclusion is due to \eqref{eq:condRintCone} and \ref{ass:qual} is verified.

The hypothesis \ref{ass:contS} (resp.~\ref{ass:contSbv}) is verified by \eqref{eq:inclJH} and invoking Lemma~\ref{lem:hausBndK} in Appendix~\ref{app:litLemmas}. Indeed, it is seen that $\cK$ being a convex polyhedral cone allows the existence of a constant $c_\cK$ such that
\[
d_{\Haus} (\cS(t_1) \cap \rge H, \cS(t_2) \cap \rge H) \le c_\cK \, \vert h(t_1) - h(t_2) \vert,
\]
where, by assumption, $h$ is a locally absolutely continuous (resp.~right-continuous BV) function.

\section{Regulation with State Feedback}\label{sec:static}

We will now use the results on the well-posedness of system~\eqref{eq:sysDVI} to solve the output regulation problem. For the sake of stability problems treated in the sequel, it will be assumed that the solutions we work with are indeed the strong solutions without any loss of generality, because one can otherwise work with strong solutions which approximate the weak solution arbitrarily closely in supremum norm.

For the control design problem treated in this section, all the states of the plant~\eqref{eq:sysDVI} and \eqref{eq:sysRef} are assumed to be available for feedback and thus a control input with static state feedback can be designed which achieves the stability and regulation.
In the formulation of our results, the following terminology is used: A quadruple of matrices $(A,B,C,D)$ is called {\em strictly passive} if there exist a scalar $\gamma>0$ and a symmetric positive definite matrix $P$ such that
\begin{equation}\label{eq:defPass}
\begin{bmatrix}
A^\top P + PA+\gamma P & PB-C^\top \\
B^\top P-C & -(D+D^\top)
\end{bmatrix}
\le 0.
\end{equation}

\begin{thm}\label{thm:static}
Consider systems~\eqref{eq:plant}, \eqref{eq:sysRef} under assumptions \ref{ass:qual}, \ref{ass:minNorm}, \ref{ass:contS}, and the output regulation error \eqref{eq:refSig}. Assume that there exist matrices $\Pi \in \R^{n \times d_r}$ and $M \in \R^{d_u \times d_r}$ such that
\begin{subequations}\label{eq:condReg}
\begin{gather}
\Pi A_r  = A \Pi + BM + F \label{eq:condRega} \\
C_r - C \Pi = 0 \label{eq:condRegb} \\
H_r - H\Pi = 0.\label{eq:condRegc}
\end{gather}
\end{subequations}
Furthermore, assume that there exists a state feedback matrix $K \in \R^{d_u \times n}$ which renders the quadruple $(\widetilde A, \widetilde G, \widetilde H, \widetilde J)$ strictly passive, where
\begin{align}
\widetilde A &:= A + B K  & \widetilde G &:= \begin{bmatrix}G & \Pi G_r \end{bmatrix} \\
\widetilde H &:= \begin{bmatrix} H \\ H \end{bmatrix} & \widetilde J &:=\begin{bmatrix} J & J_r\\J & J_r\end{bmatrix}.
\end{align}
Then the output regulation problem is solvable with the following static feedback control law:
\begin{equation}\label{eq:staticControl}
u(t) = Kx(t) + (M-K\Pi) x_r(t).
\end{equation}
\end{thm}

\begin{rem}
In the work of \cite{Francis77}, the control law \eqref{eq:staticControl} was proposed to solve the output regulation problem in LTI systems, where $\Pi$ and $M$ were obtained as solution of \eqref{eq:condRega}, \eqref{eq:condRegb} only, and $K$ is any matrix that makes $(A+BK)$ Hurwitz.
The strict passivity requirement, and additional conditions on the matrix $\Pi$ in \eqref{eq:condRegc}  are required in the well-posedness and stability analyses for the class of systems considered in this paper.
\end{rem}

\begin{rem}
The intuition behind the stability analysis is that the matrix $\Pi$ embeds the reference signal $x_r$ in the state space $\R^n$ to which the state variable $x$ converges, that is, the control action regulates the error variable $e:=x - \Pi x_r$ to the origin. The internal model principle (conditions \eqref{eq:condRega}, \eqref{eq:condRegb}) shapes the closed-loop dynamics to match $\Pi \dot x_r$.
If we let $\widetilde \eta :=(\eta^\top, -\eta_r^\top)^\top$, then
the additional conditions allow us to make the error dynamics passive with input $\widetilde G \widetilde\eta$ and the auxiliary output $\widetilde H e + \widetilde J \widetilde \eta$.
The monotonicity of the normal cone operator then leads to the negative definiteness of the derivative of the Lyapunov function $e^\top P e$.
\end{rem}

\begin{figure}[!t]
\centering
\begin{tikzpicture}
\draw (2,0) node [rectangle, draw, minimum height =0.75cm, text centered] (q1) {$\dot e = \widetilde A e+\widetilde G\widetilde\eta$};
\draw (2,-2) node [rectangle, draw, minimum height =0.75cm, text centered] (q2) {$\begin{aligned} \eta &\in -\cN_{\cS}(Hx+J\eta) \\ \eta_r & \in -\cN_{\cS}(H_rx_r+J_r\eta_r)\end{aligned}$};

\coordinate (p1) at ($(q1.base east)+(1.5,0)$);
\coordinate (p2) at ($(q2.west)+(-0.75,0)$);
\draw [thick, ->] (q1.base east) -- (p1) -- (p1 |- q2.east) node[anchor=south west]{$Hx,H_rx_r$} -- (q2.east);
\draw [thick, ->] (q2.west) -- (p2) -- (p2 |- q1.base west) -- (q1.base west) node[anchor=south east]{$\widetilde \eta \quad$};
\draw [thick, ->] (q1.10) -- ++(1.5,0) node[anchor = west] {$\widetilde y = \widetilde H e + \widetilde J \widetilde\eta$};
\end{tikzpicture}
\caption{Passivity based interpretation of the error dynamics.}
\label{fig:intPass}
\end{figure}
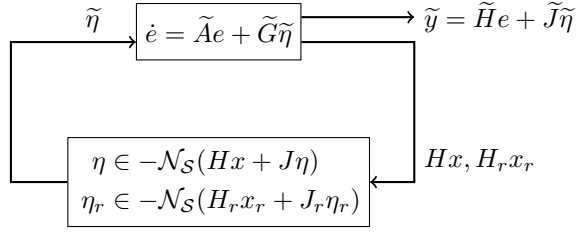

\begin{myproof}{of Theorem~\ref{thm:static}}
The proof consists of two parts: the first one shows existence of unique solutions, and the second part analyses the convergence of the error dynamics.

{\em a) Well-posedness:} With control input \eqref{eq:staticControl}, the closed-loop system is written as
\begin{gather*}
\dot x(t) = (A+BK) x(t) + (F+ BM-BK\Pi )x_r(t) +G \eta(t) \\
\eta  \in -\cN_{\cS(t)}(Hx(t)+J\eta(t)).
\end{gather*}
Since $(\widetilde A,\widetilde G,\widetilde H, \widetilde J)$ is assumed to be strictly passive, it follows that $(\widetilde A, G,H,J)$ is also strictly passive because, using \eqref{eq:defPass}, the matrix of  interest is a sub-block of the negative semi-definite matrix associated with $(\widetilde A,\widetilde G,\widetilde H, \widetilde J)$, and is hence negative-semidefinite.
Thus, the matrix $J$ is positive semidefinite and $\ker (J+J^\top) \subseteq \ker(PG-H^\top)$, see \cite{CamlIann14} for proof.
All the remaining hypothesis of Theorem~\ref{thm:solMain} hold by construction, and hence the closed-loop system has a unique solution.

{\em b) Regulation:} Let $\Pi$ be a matrix that satisfies \eqref{eq:condReg} and introduce the variable $e := x - \Pi x_r$. The output regulation is achieved if we can show that $\lim_{t\rightarrow \infty} e(t) = 0$, since
\begin{align*}
w(t) & = C x (t) -C_r x_r(t) \\
& = C x(t) - C\Pi x_r(t) = C e(t).
\end{align*}
To show that $e(t) \rightarrow 0$ as $t \rightarrow \infty$, we observe that
\begin{align*}
\dot e(t) & = (A+BK)x(t) + (F+BM - BK\Pi) x_r(t) 
+G \eta(t) - \Pi A_r x_r(t) - \Pi G_r \eta_r(t) \\
& = (A+BK)x(t) - (\underbrace{\Pi A_r - F - B M}_{=A\Pi} +BK\Pi)x_r (t) 
 + G \eta(t) - \Pi G_r \eta_r(t) \\
& = (A+BK)x(t) - (A+BK)\Pi x_r(t) + \widetilde G \widetilde\eta
\end{align*}
where we recall that $\widetilde \eta := (\eta^\top, -\eta_r^\top)^\top$.
Now introduce the Lyapunov function $V(e) = e^\top P e$, so that the following holds for almost all $t \ge 0$ :
\begin{align*}
\dot V(e(t)) & = e(t)^\top \left((A+BK)^\top P+ P(A+BK) \right) e(t) 
 + 2 e(t)^\top P \widetilde G \widetilde\eta(t) \\
& \le - \gamma e(t)^\top P e(t) + 2 e(t)^\top \widetilde H^\top \widetilde\eta (t)+\widetilde\eta(t)^\top(\widetilde J+\widetilde J^\top)\widetilde\eta \\
& = - \gamma e(t)^\top P e(t) + 2 \langle \eta(t), H e(t) + J \eta(t) - J_r \eta_r(t)\rangle 
 - 2 \langle \eta_r(t), H e(t) + J \eta(t) - J_r\eta_r(t)\rangle
\end{align*}
where the inequality was obtained by using the passivity of $(A+BK,\widetilde G,\widetilde H,\widetilde J)$.
Recalling that $He(t) = Hx(t) - H\Pi x_r(t) = Hx(t) - H_r x_r(t)$, the monotonicity of the normal cone operator leads to, for each $t \ge 0$,
\[
\begin{aligned}
\langle \eta(t), Hx(t) + J \eta(t) - H_r x_r (t) - J_r \eta_r(t) \rangle & \le 0 \\
\langle \eta_r(t), H x(t) + J \eta(t) -H_r x_r(t) - J_r \eta_r(t) \rangle & \ge 0.
\end{aligned}
\]
It thus follows that $\dot V(e(t)) < 0$ for all $e(t) \neq 0$, and thus $e(t) \rightarrow 0$ as $t \rightarrow \infty$.

{\em Closed-loop stabilization:} To show that $x$ converges to zero when $x_r\equiv 0$, we observe that asymptotic convergence of $e(t) = x(t) -\Pi x_r(t)$ to the origin has already been proven. In particular, when $x_r \equiv 0$, the state $x(t)$ converges to the origin asymptotically.
\end{myproof}

\begin{rem}
One can also generalize the result of Theorem~\ref{thm:static} to the case where the variation in the set-valued map $\cS$ is governed by a locally BV function as in \ref{ass:contSbv}. In this case, it must be verified that, at the discontinuities of $x$ or $x_r$, the Lyapunov function $V(e)$ is also decreasing, that is, $e^{+\top} P e^+ \le e^{-\top} P e^-$, where $e^-, e^+$ denote the value of $e$ before and after a jump instant, respectively. We can write $e^+ - e^- = \widetilde G\widetilde \eta$, so that
\[
e^{-\top}Pe^- = e^{+\top}Pe^+ -2e^{+\top}P\widetilde G \widetilde \eta + \widetilde \eta^\top \widetilde G^\top P \widetilde G \widetilde \eta \ge e^{+\top}Pe^+
\]
where the middle term is nonnegative due to the same arguments as used in the proof of Theorem~\ref{thm:static}. The formal arguments to deduce asymptotic stability for systems with BV solutions can be found in \cite{TanwBrog14a, TanwBrog15, TanwBrog16}.
\end{rem}


\section{Error Feedback and Dynamic Compensator}\label{sec:dyn}

In this section, it will no longer be assumed that the states $x(\cdot)$ and $x_r(\cdot)$ are available for feedback, but only the regulation error $w(\cdot)$ is available to the controller.
Our approach is based on the {\em certainty equivalence principle} where we first design an estimator for the state variables $x(\cdot)$ and $x_r(\cdot)$, and then define the control law as a function of these estimates.

Toward this end, the estimator we propose is defined as follows:
\begin{subequations}\label{eq:obsDyn}
\begin{equation}\label{eq:obsDyna}
\begin{pmatrix} \dot{\hat x} \\ \dot{\hat x}_r \end{pmatrix} = \begin{bmatrix} \begin{pmatrix} A & F \\ 0 & A_r \end{pmatrix} - \begin{pmatrix} L_0 \\ L_1 \end{pmatrix} \begin{pmatrix} C & -C_r \end{pmatrix} \end{bmatrix} \begin{pmatrix} \hat x \\ \hat x_r \end{pmatrix} + \begin{pmatrix} L_0 \\ L_1 \end{pmatrix} w 
+ \begin{pmatrix} B \\ 0 \end{pmatrix} u + \begin{pmatrix} G & 0\\ 0& G_r \end{pmatrix} \begin{pmatrix} \hat \eta \\ \hat \eta_r \end{pmatrix}
\end{equation}
\begin{equation}\label{eq:obsDynb}
\begin{pmatrix} \hat \eta \\ \hat \eta_r \end{pmatrix} \in -\cN_{\cS(t) \times \cS(t)} \left(\widehat H \begin{pmatrix}\hat x \\ \hat x_r \end{pmatrix}
+\widehat J \begin{pmatrix} \hat \eta \\ \hat \eta_r \end{pmatrix}
\right),
\end{equation}
\end{subequations}
where the gain matrix $L:=\begin{pmatrix}L_0 \\ L_1\end{pmatrix} $ will be designed in the sequel, and the matrices $\widehat H$, $\widehat J$ are defined as follows:
\[
\widehat H := \begin{pmatrix} H & 0\\ 0& H_r \end{pmatrix}, \quad
\widehat J := \begin{pmatrix} J & 0\\ 0& J_r \end{pmatrix}.
\]
For brevity, we have suppressed the time argument in \eqref{eq:obsDyn}, and will do so in the remainder of this section unless required. Let us also introduce the following notation:
\begin{gather*}
\widehat A := \begin{pmatrix} A & F \\ 0 & A_r \end{pmatrix}, \quad
\widehat C:= \begin{pmatrix} C & -C_r \end{pmatrix}, \quad
\widehat G := \begin{pmatrix} G & 0\\ 0& G_r \end{pmatrix}.
\end{gather*}

\begin{thm}
Consider systems~\eqref{eq:plant}, \eqref{eq:sysRef} under assumptions \ref{ass:qual}, \ref{ass:minNorm} and \ref{ass:contS}, and the output regulation error \eqref{eq:refSig}.
Suppose that there exist a feedback matrix $K$ and an injection matrix $L$ that render the quadruples $(\widetilde A,\widetilde G,\widetilde H, \widetilde J)$ and $(\widehat A-L\widehat C,\widehat G, \widehat H,\widehat J)$ strictly passive, respectively.
If there exist matrices $\Pi \in \R^{n \times d_r }$ and $M \in \R^{d_u \times d_r}$ that satisfy \eqref{eq:condReg}, then the output regulation problem is solved by letting
\begin{equation}\label{eq:dynControl}
u(t) = K \hat x(t) + (M-K\Pi) \hat x_r(t).
\end{equation}
\end{thm}

\begin{proof}
If we denote the state of the closed-loop system by $X:= \col (x, \hat x, \hat x_r)$, and let $\Lambda:= \col (\eta, \hat \eta, \hat \eta_r)$, then the control law \eqref{eq:dynControl} results in the following dynamics:
\begin{subequations}\label{eq:sysCl}
\begin{gather}
\dot X = A_{cl} X + F_{cl} x_r + G_{cl} \Lambda \label{eq:sysCla} \\
\Lambda \in -\cN_{\cS(t) \times \cS(t) \times \cS(t)} \left( H_{cl} X + J_{cl} \Lambda \right), \label{eq:sysClb}
\end{gather}
\end{subequations}
where the matrices $A_{cl}$, $F_{cl}$ are obtained from \eqref{eq:plant}, \eqref{eq:sysRef}, \eqref{eq:obsDyn}, \eqref{eq:dynControl}, and 
\[
G_{cl}:=\begin{bmatrix} G & 0 \\ 0 & \widehat G \end{bmatrix}, \quad
H_{cl}:=\begin{bmatrix} H & 0 \\ 0 & \widehat H \end{bmatrix}, \quad
J_{cl}:=\begin{bmatrix} J & 0 \\ 0 & \widehat J \end{bmatrix}.
\]

{\em Well-posedness}: To show that the closed-loop system admits a unique solution, we follow the same procedure as in the proof of Theorem~\ref{thm:static}. That is, we just need to find a symmetric positive definite matrix $P_{cl}$ such that $ \ker (J_{cl}+J_{cl}^\top) \subseteq \ker (P_{cl}G_{cl} - H_{cl}^\top)$, as all other hypotheses of Theorem~\ref{thm:solMain} hold by construction. The matrix $J_{cl}$ is positive semidefinite by construction.
To find the matrix $P_{cl}$, it is noted that, by assumption, there exist symmetric positive definite matrices $P,\widehat{P}$ such that the LMI \eqref{eq:defPass} holds for the quadruples $(A+BK, G, H, J)$ and $(\widehat A-L\widehat C, \widehat G, \widehat H, \widehat J)$, respectively, for some dissipation constants $\gamma, \widehat \gamma$.
It thus holds that $\ker(J+J^\top) \subseteq \ker(PG-H^\top)$, and $\ker(\widehat J+\widehat J^\top) \subseteq \ker(\widehat P\widehat G - \widehat H^\top)$.
Now, simply choose
\[
P_{cl} := \begin{bmatrix} P & 0 \\ 0 & \widehat P \end{bmatrix},
\]
then it is easily checked that $P_{cl}$ is symmetric, positive definite, and $\ker (J_{cl}+J_{cl}^\top) \subseteq \ker(P_{cl} G_{cl} - H_{cl}^\top) $.

{\em Regulation:}
We are interested in asymptotic stability of the origin for the variable $e$, defined as
\begin{equation}
e := \begin{pmatrix} x -\Pi x_r \\ x - \hat x \\ x_r - \hat x_r \end{pmatrix} 
=: \begin{pmatrix} e_x \\ \tilde x \\ e_{r} \end{pmatrix}
=: \begin{pmatrix} e_x \\ e_\xi \end{pmatrix}.
\end{equation}
By introducing the matrix $W:=\begin{bmatrix}-K & (M-K\Pi)\end{bmatrix}$, it can be shown using the equations in \eqref{eq:condReg} that
\begin{equation}\label{eq:regDyn}
\dot e = \begin{bmatrix} A + BK & BW \\ 0 & \widehat A - L \widehat C \end{bmatrix} \begin{pmatrix} e_x \\ e_\xi \end{pmatrix} 
 + \begin{bmatrix} \widetilde G & 0 & 0 \\ 0 & G & 0 \\ 0 & 0 & G_r\end{bmatrix} \begin{pmatrix} \widetilde \eta \\ \eta - \hat \eta \\ \eta_r - \hat \eta_r \end{pmatrix}
\end{equation}
where we recall that $\widetilde G= [G~\Pi G]$, and $\widetilde \eta = \col (\eta, -\eta_r)$. Also, $\eta,\eta_r$ and $\col (\hat \eta, \hat \eta_r)$ are defined in \eqref{eq:plantb}, \eqref{eq:sysRefb} and \eqref{eq:obsDynb}, respectively.

Consider the Lyapunov function $V(e) = e^\top \begin{bmatrix} \alpha P & 0 \\ 0 & \beta \widehat P \end{bmatrix} e$ for some $\alpha, \beta > 0$ to be specified later.
The derivative of $V(\cdot)$ along the trajectories of the closed-loop system \eqref{eq:sysCl} satisfies the following relations:
\begin{align*}
\dot V(e) & = \alpha e_x^\top (P (A+BK) + (A+BK)^\top P) e_x  
  +2\alpha e_x^\top PBW e_\xi \\
& \quad + e_\xi^\top (\beta \widehat P (\widehat A - L \widehat C) + (\widehat A- L \widehat C)^\top \beta \widehat P)e_\xi 
+ 2\alpha  e_x^\top P \widetilde G \widetilde \eta + 2\beta e_\xi^\top \widehat P \,\widehat G  
\begin{pmatrix} \eta - \hat \eta \\ \eta_r - \hat \eta_r \end{pmatrix} \\
& \le - \alpha \gamma e_x^\top P  e_x - \beta \widehat \gamma e_\xi^\top \widehat P e_\xi + 2\alpha e_x^\top PBW e_\xi 
 +2 \alpha\widetilde\eta^\top(\widetilde H e_x+ \widetilde J \widetilde \eta) \\
& \quad +2 \beta (\eta-\hat \eta)^\top(H\tilde x+J(\eta-\hat \eta)) 
 +2 \beta (\eta_r-\hat \eta_r)^\top(H_re_r+J_r(\eta_r-\hat \eta_r))
\end{align*}
where the last inequality is obtained using the passivity assumption on $(A+BK,G,H,J)$ and $(\widehat A-L\widehat C, \widehat G, \widehat H, \widehat J)$.
Recalling that $H\Pi = H_r$, it now follows that
\begin{align*}
\widetilde \eta^\top(\widetilde H e_x+\widetilde J\widetilde\eta) & = 
\langle \eta(t), H e(t) + \widetilde J \eta(t) - J_r \eta_r(t)\rangle - \langle \eta_r(t), H e(t) + J \eta(t) - J_r\eta_r(t)\rangle \\
& = \langle \eta(t), Hx(t) + J \eta(t) - (H_r x_r(t) + J_r \eta_r(t)) \rangle - \langle \eta_r(t), Hx(t) + J \eta(t) - (H_r x_r(t) + J_r \eta_r(t)) \rangle\\
& \le 0
\end{align*}
where the inequality follows due to the definition of the normal cone as $Hx(t) + J \eta(t) \in \cS(t)$, and $H_rx_r(t) + J_r \eta_r (t) \in \cS(t)$, for each $t \ge 0$, due to the adopted solution concept.
Using similar arguments,
\begin{align*}
& \quad (\eta-\hat \eta)^\top(H\tilde x+J(\eta-\hat \eta)) \\
& =  \langle \eta(t), Hx(t) + J \eta(t) - (H \hat x (t) + J \hat \eta(t)) \rangle - \langle \hat \eta(t), Hx(t) + J \eta(t) - (H \hat x(t) + J \hat \eta(t)) \rangle\\ 
& \le 0
\end{align*}
and
\begin{align*}
& \quad (\eta_r-\hat \eta_r)^\top(H_re_r+J_r(\eta_r-\hat \eta_r)) \\
& =  \langle \eta_r(t), H_rx_r(t) + J_r \eta_r(t) - (H_r \hat x_r (t) + J_r \hat \eta_r(t)) \rangle - \langle \hat \eta_r(t), H_rx_r(t) + J_r \eta_r(t) - (H_r \hat x_r(t) + J_r \hat \eta_r(t)) \rangle\\ 
& \le 0
\end{align*}
Plugging these relations in the expression for $\dot V$, and using the notation $\sigma_{\min}$ to denote the smallest eigenvalue of a matrix, and $\chi$ to denote the induced Euclidean matrix norm of $PBW$, we get
\begin{align*}
\dot V (e) & \le -\alpha \gamma \, \sigma_{\min}(P) \, |e_x|^2 -\beta \widehat{\gamma}\sigma_{\min}\, (\widehat P) \, |e_\xi|^2 
 + 2 \alpha \chi \, |e_x| \, |e_\xi| \\
& \le  -\alpha \gamma\, \sigma_{\min}(P) \, |e_x|^2 \!\! -\beta \widehat{\gamma} \sigma_{\min}\, (\widehat P) \, |e_\xi|^2 \!\!+ |e_x|^2 
 + \alpha^2 \chi^2 |e_\xi|^2 \\
&= - (\alpha\gamma\, \sigma_{\min}(P) - 1) |e_x|^2 
- (\beta \widehat \gamma \, \sigma_{\min}(\widehat P) - \alpha^2\chi^2) |e_\xi|^2.
\end{align*}
Thus, choosing $\alpha, \beta$ in the definition of the Lyapunov function $V(\cdot)$ such that $\alpha\, \gamma \sigma_{\min}(P) > 1$ and $\beta \widehat{\gamma}\, \sigma_{\min}(\widehat P) > \alpha^2\chi^2)$, makes $\dot V$ negative definite.
In particular $e_x$ converges to zero, from which it follows that $w = Cx - C_r x_r = C (x - \Pi x_r) = Ce_x$ converges to zero.

{\em Closed-loop stabilization:} Asymptotic stability of the origin of the unforced closed-loop system also follows directly. Indeed, when $x_r \equiv 0$, $e(t) = \col(x(t), \tilde x(t), \hat x_r(t))$ converges to the origin asymptotically.
\end{proof}

\section{Applications}\label{sec:app}

Let us revisit the examples that were pointed out at the beginning of the article as a motivation for studying the problem of output regulation for the proposed class of systems.

\subsection{Viability of Convex Polytopes with Set-Valued Control}

In this section, we apply our theoretical results to formulate an output regulation problem with polyhedral constraints.
In particular, the plant to be controlled is now described by:
\begin{equation}
\dot x (t) = A x(t) + Bu(t) + Fx_r(t).
\end{equation}
and our aim is to find a control input $u$ such that the following two objectives are met:
\begin{obj}[leftmargin=3em]
\item \label{obj:lcs1} For a given closed convex polyhedral cone $\cK \subseteq \R^{d_s}$, and a locally absolutely continuous function $h(\cdot):[0,\infty) \rightarrow \R^{d_s}$, the state $x$ satisfies the constraint that $y(t):=H x(t) + h(t) \in \cK$, for all $t \ge 0$, where the matrix $H \in \R^{d_s \times n}$ satisfies the constraint \eqref{eq:consH}.
\item \label{obj:lcs2} the output regulation is achieved with respect to the variable $w = Cx - C_r x_r$ for a given exogenous signal $x_r$.
\end{obj}
Thus, in addition to the output regulation, we have a viability problem in Aubin's sense \cite[Chapter 11]{AubiBaye11} since the state of the plant is constrained to evolve within a predefined polyhedral set at all times.

In order to achieve these objectives, the exosystem generating the reference trajectories is expressed in the form of a linear complementarity system (LCS), that is, we consider:
\begin{subequations}\label{eq:refLCS}
\begin{gather}
\dot x_r(t) = A_r x_r(t) + G_r \eta_r \label{eq:refLCSa}\\
 -\cK^* \ni \eta_r \perp H_r x_r(t) + h(t) \in \cK, \label{eq:refLCSb}
\end{gather}
\end{subequations}
and hence the results of this section can also be interpreted in the context of trajectory tracking for a class of trajectories generated by LCS.

It will be assumed that system \eqref{eq:refLCS} is initialized such that it admits a solution and the resulting state trajectory $x_r$ is locally absolutely continuous. By letting, 
\begin{equation}\label{eq:defS}
\cS(t):= \{z\in \R^{d_s} \,| \, z + h(t) \ge 0\},
\end{equation}
system \eqref{eq:refLCS} is equivalently written as:
\begin{gather*}
\dot x_r(t) = A_r x_r(t) + G_r \eta_r \\
\eta_r \in - \cN_{\cS(t)}(H_r x_r(t)).
\end{gather*}

For the plant under consideration, we split the control input $u$ as $u:= u_{\reg} + u_\eta$ with the motivation that $u_{\reg}$ refers to the classical internal model control input of the form \eqref{eq:staticControl} and $u_\eta$ would play the role of normal vector required to maintain the constraint (a viability controller).
Thus, the following control system is realized:
\begin{subequations}\label{eq:contLCS}
\begin{gather}
\dot x (t) = A x(t) + Bu_{\reg}(t) + Bu_\eta(t) + Fx_r(t), \\
u_\eta(t) \in -\cN_{\cS(t)}(Hx(t))\label{eq:contLCSb}\\
u_{\reg}(t) = K x + (M-K\Pi) x_r(t)
\end{gather}
\end{subequations}
where $M$ and $\Pi$ are obtained as solutions to \eqref{eq:condReg} by taking $G=B$, and $K$ is chosen such that $(A+BK, B, H)$ is strictly passive.
The following result then readily follows from Corollary~\ref{cor:solLCCS} and Theorem~\ref{thm:static}.
\begin{cor}\label{cor:viab}
The system~\eqref{eq:contLCS} possesses a unique solution and the resulting state trajectory $x$ satisfies the properties \ref{obj:lcs1} and \ref{obj:lcs2}.
\end{cor}

It may be the case that equation~\eqref{eq:contLCSb} is not easily computable for arbitrary closed convex cone $\cK$. For the case, when $\cK = \R_+^{d_s}$, the positive orthant, we replace~\eqref{eq:contLCSb} with an alternate expression in the form of complementarity relation for which there exist efficient solvers.
In that case, it may be verified that
\begin{align*}
& \phantom{\Leftrightarrow} \quad u_\eta(t) \in -\cN(\cS(t); Hx(t)) \\
& \Leftrightarrow 0\le u_\eta(t) \perp y(t) = Hx(t)+ h(t)\ge 0 \\
&\Leftrightarrow
\begin{cases}
u_\eta(t) = 0 & \text {if } Hx(t) \in \text{int } \cS(t) \\
0 \le u_\eta (t) \perp \dot y(t)\ge 0  & \text {if } Hx(t) \in  \text{bd } \cS(t)
\end{cases}
\end{align*}
where $\dot y = HAx + HBu_{\reg} + HBu_\eta+HFx_r + \dot h$ is considered to be known. If the matrix $HB$ is positive definite, then there exists a unique solution to the aforementioned LCP. Existence (without uniqueness) is also guaranteed under certain conditions, see several results in \cite[Section~5.4.2]{Brog16} and \cite{CottPang92} on the solvability of LCP, in particular for the case when $HB$ is positive semidefinite.

%
%

\begin{exam}
As an illustration of Corollary~\ref{cor:viab}, let us consider a second order LTI system described by the equations
\[
\dot x_1 = -0.1 x_1+ x_2; \quad \dot x_2= u.
\]
The exosystem is defined as the following linear complementarity system:
\begin{gather}
\dot x_r := \begin{pmatrix} \dot x_{r1} \\ \dot x_{r2} \end{pmatrix} = \begin{bmatrix} -0.1 & 1\\-2 & 1 \end{bmatrix} x_r + \begin{bmatrix} 0 & 0 \\ -1 & 1\end{bmatrix} \eta_r \notag\\
0 \le \eta_r \bot \begin{pmatrix} -x_{r2} \\ x_{r2} \end{pmatrix} + \begin{pmatrix} 1 \\ 1 \end{pmatrix} \ge 0. \label{eq:compEx}
\end{gather}
Consider the set $\cS:=\{z \in \R \,\vert \, z+1 \ge 0\}$ and the matrix $H:=[{0\atop 0}{-1\atop 1}]$, the relation \eqref{eq:compEx} is equivalently expressed as $\eta_r \in -\cN_{\cS\times \cS} (Hx_r)$.
We are interested in designing a control input $u$, such that $\lim_{t \rightarrow \infty} |x_2(t) - x_{r2}(t)| = 0$, and $\forall \, t \ge 0$, $|x_2(t)| \le 1$, or equivalently $Hx(t) \in \cS \times \cS$.
Verbally speaking, the exosystem has been chosen so that the plot of $x_{r2}$ (versus time) resembles a sine wave clipped at the value 1, see Fig.~\ref{fig:clipSin}.
The control objective is to guarantee $|x_2(t)| \le 1$ and that $x_2$ converges asymptotically to $x_{r2}$.
Decomposing the input as $u:= u_\reg+u_\eta$ results in the closed-loop system of the form~\eqref{eq:contLCS}.
In the notation of Theorem~\ref{thm:static}, we let $\Pi = I_{2\times 2}$, $P= [{2 \atop 0}{0 \atop 1}]$, $K = [-2~-2]$, and $M= [-2~1]$, so that
$u_\reg(t):= -2x_1(t) -2x_2(t) +3 x_{r2}(t)$
follows from~\eqref{eq:staticControl}.
The discontinuous component of the input\footnote{We are implicitly using the fact that the two constraints imposed in this problem, $x_2 \le 1$ and $x_2 \ge -1$, are not active simultaneously.
Thus, the complementarity formulation~\eqref{eq:compClipSin} ensures that $u_{\eta} = -u_{\eta1}$ if $x_2 = 1$, and $u_\eta=u_{\eta2}$ if  $x_2 = -1$, otherwise $u_\eta = 0$.}
$u_\eta := -u_{\eta 1} + u_{\eta 2}$ is obtained as a solution of the following linear complementarity problems
\begin{equation}\label{eq:compClipSin}
\begin{aligned}
0 \le u_{\eta 1} \perp - x_2 + 1 \ge 0 \quad \Longleftrightarrow
\begin{cases}
u_{\eta1}(t) = 0 & \text{if } x_2(t) < 1\\
0 \le u_{\eta 1}(t) \perp -u_{\reg}(t) + u_{\eta 1}(t) \ge 0 & \text{if } x_2(t) = 1
\end{cases} \\
0 \le u_{\eta 2} \perp x_2 + 1 \ge 0 \quad \Longleftrightarrow
\begin{cases}
u_{\eta2}(t) = 0 & \text{if } x_2(t) > -1\\
0 \le u_{\eta 2}(t) \perp u_{\reg}(t) + u_{\eta 2}(t) \ge 0 & \text{if } x_2(t) = -1
\end{cases}
\end{aligned}
\end{equation}
which is always guaranteed to exist~\cite{CottPang92}.
The simulations for this example were carried out using the {\sc siconos} platform\footnote{Please visit {http://siconos.gforge.inria.fr/4.1.0/html/index.html} for details.} developed at INRIA.
The results of the simulation are shown in Fig.~\ref{fig:clipSin}.
\begin{figure}
\centering
\includegraphics[width=0.75\textwidth]{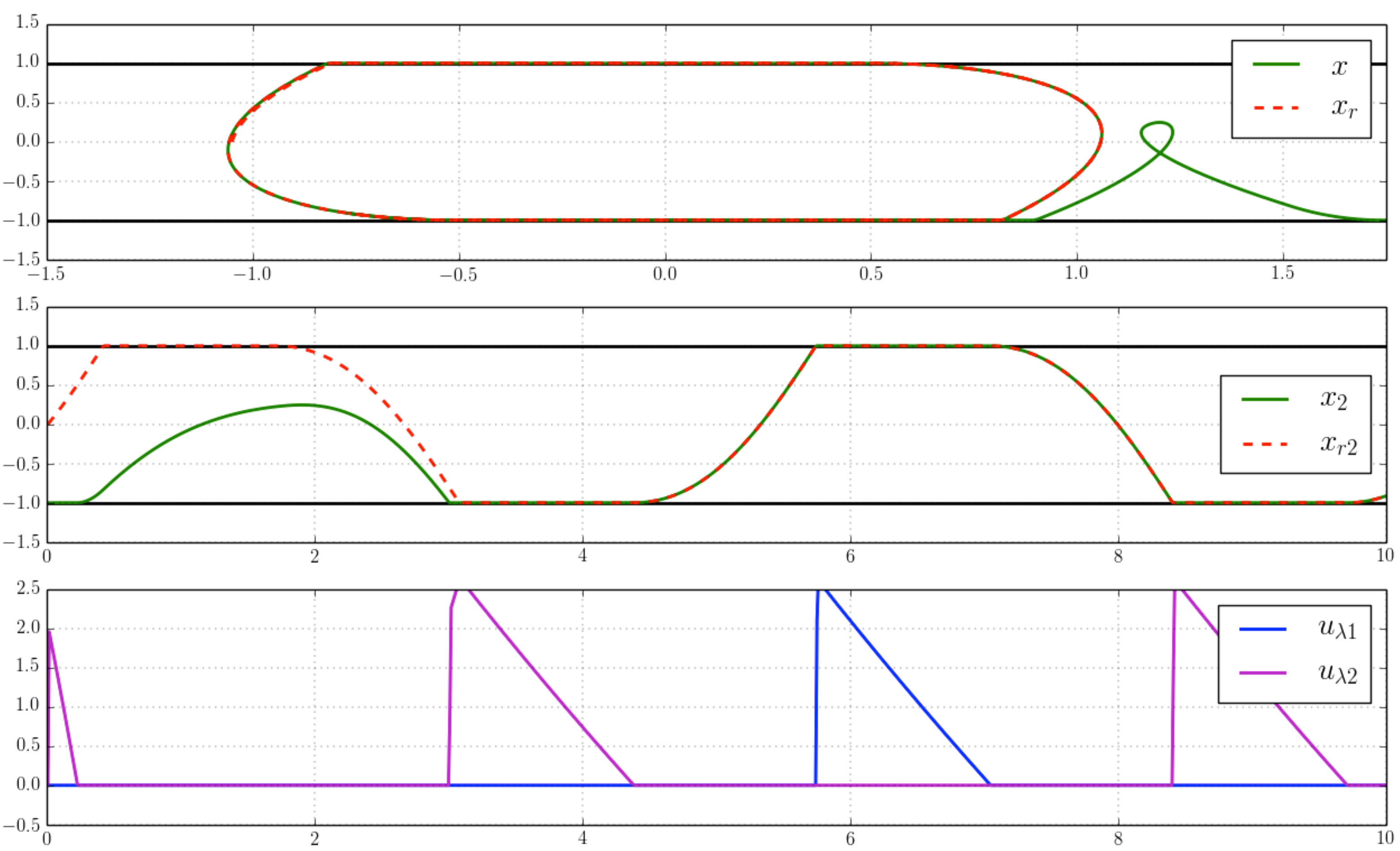}
\caption{The top plot shows the phase portrait of the trajectories of the plant and the exosystem. The middle plot confirms that $x_2$ converges to $x_{r2}$ while staying within the set $\cS$. The bottom plot shows the values of discontinuous component of the control input $u_\eta$, which only become nonzero when $x$ is on the boundary of the set $\cS$.}
\vskip -1em
\label{fig:clipSin}
\end{figure}
\end{exam}

\subsection{Regulation in Power Converters}

The framework of EVIs is extremely useful for modeling electrical circuits with switches, diodes, transistors, etc., which often induce discontinuities in state variables of the systems during their mode of operation.
In particular, it provides a very efficient framework for the numerical simulation \cite[Chapters 7,8]{AcarBonn11}.
Designing state feedback algorithms such that certain variables of the system follow a prescribed trajectory is seen as a challenging task due to presence of discontinuities in the system description. The results of this paper allow us to solve this problem for certain class of circuits and here we discuss the direct application of these results. It remains to be seen to what extent these results can be generalized to explore further similar applications in the domain of power converters, especially, regulations problems which involve the design of state-dependent switching laws.

The class of electrical circuits we consider here is assumed to be modeled by the following equations \cite{VascCaml09}:
\begin{subequations}\label{eq:sysElecApp}
\begin{gather}
\dot x(t) = A x(t) + F x_r(t) + G \eta(t) +  Bu(t) + B_{\text{ext}} f_{\text{ext}}(t) \label{eq:sysElecAppa} \\
v(t) = Hx(t) + J \eta(t) + h(t) \label{eq:sysElecAppb}\\
\cK \ni v(t) \perp \eta(t) \in \cK^*, \label{eq:sysElecAppc}
\end{gather}
\end{subequations}
where $u$ is a control input, $f_{\text{ext}}$ is a locally essentially bounded measurable function, and $h$ is a locally BV function. The reason for including this extra forcing term is that, very often, the power converters are driven by a voltage or current source which must be taken into account. We limit ourselves to the case where $\cK$ is a closed polyhedral convex cone, and $\cK^*$ denotes the dual of $\cK$.
The class of reference signals that we consider is described as:
\begin{subequations}\label{eq:sysRefElec}
\begin{gather}
\dot x_r(t) = A_r x_r(t) + G_r \eta_r(t) + B_r f_{\text{ext}}(t) \\
v_r(t) = H_rx_r(t) + J_r\eta_r(t) + h_r(t) \\
\cK \ni v_r(t) \perp \eta_r(t) \in \cK^*.
\end{gather}
\end{subequations}

\begin{cor}
Assume that there exist matrices $\Pi \in \R^{n \times d_r}$ and $M \in \R^{d_u \times d_r}$ such that the regulator conditions \eqref{eq:condReg} hold, and there is a feedback matrix $K$ which renders the quadruple $(\widetilde A, \widetilde G, \widetilde H, \widetilde J)$ strictly passive.
Furthermore, assume that there exists a matrix $N$ such that
\begin{equation}\label{eq:condfext}
B N + B_{\text{\em ext}} =  \Pi B_r
\end{equation}
and $h = h_r$, then the control input
\[
u(t) = K x(t) + (M-K\Pi) x_r(t) + Nf_{\text{\em ext}}(t)
\]
solves the output regulation problem.
\end{cor}

The additional condition \eqref{eq:condfext} and term $Nf_{\text{ext}}(t)$ in the control input allow us eliminate the exogenous signal $f_{\text{ext}}$ in the error dynamics. The stability of the error dynamics is then analyzed using the same arguments as given in the proof of Theorem~\ref{thm:static}.

\begin{exam}

%
%
%
%
%
%
%
%
%

Consider the circuit given in Figure~\ref{fig:exCircuit} in which the voltage-current characteristic of a diode is most appropriately modeled using an EVI. We consider the state variables of the system to be $x_1(t) := \int_0^ti_1(s) \, ds$, and $x_2(t) := i_2(t)$. We denote by $\eta$ the voltage drop across the diode. The voltage source $f_{\text{ext}}$ is a priori fixed as a locally BV function of time. The dynamics of the circuit are governed by the following equations:
\begin{gather*}
\dot x(t) = \begin{bmatrix} \frac{-1}{RC} & -1\\ \frac{1}{LC} & 0 \end{bmatrix} x(t) + \begin{bmatrix} \frac{1}{R} \\ 0 \end{bmatrix} \eta(t) + \begin{bmatrix} 0 \\ \frac{-1}{L}\end{bmatrix} u(t) + \begin{bmatrix} \frac{1}{R} \\ 0\end{bmatrix} f_{\text{ext}}(t) \\
0 \le \eta(t) \perp \begin{bmatrix} \frac{-1}{RC} & 0\end{bmatrix} x(t) + \frac{1}{R} \eta(t) +\frac{1}{R} f_{\text{ext}}(t)\ge 0,
\end{gather*}
where the complementarity conditions represent the voltage-current characteristic of a diode \cite{AcarBonn11}.
The control input $u$ is to be designed via state feedback such that the voltage across the capacitor $x_1(t)$ has the same value as $f_{\text{ext}}(t)$.
To do so, we define the reference system as follows:
\begin{gather*}
\dot x_r(t) = a_r x_r(t) + g_r \eta_r(t) + b_r f_{\text{ext}}(t) \\
0 \le \eta_r(t) \perp a_r x_r(t) + g_r \eta_r(t) + b_r f_{\text{ext}}(t) \ge 0,
\end{gather*}
where we take $a_r = -\frac{1}{RC}$, $g_r = \frac{1}{R}$, $b_r = \frac{1}{R}$.

To solve the regulation problem, we take $\Pi = \begin{bmatrix} 1 \\ 0\end{bmatrix}$, and the resulting control input is
\[
u(t) = k_1 x_1(t) + k_2 x_2(t) + \frac{1}{C} x_r(t) -k_1x_r(t).
\]
where we let $K:=[k_1~k_2]$, and the constants $k_1,k_2$ need to be computed.
To carry the computations, we chose $L=1\text{mH}$, $C = 10\text{mF}$, $R = 10 \, \Omega$.
The LMI solvers in MATLAB were used to find a positive definite matrix $P$ and the gain vector $K$ such that the criteria of Theorem~1 holds.
One possible choice of $P$ and $K$ obtained from these solvers is
\[
P = \begin{bmatrix}2240.9 & -4.4029 \\ -4.4029 & 0.0137 \end{bmatrix} , \quad K = \begin{bmatrix} -1000 & 5\end{bmatrix}.
\]
We then used the solvers programmed in the INRIA {\sc siconos} platform to simulate the resulting system of differential equations and the variational inequalities. For the simulations, we chose $f_{\text{ext}}(t) = \lfloor 10 \cdot t\rfloor$, which results in the plots reported in Figure~\ref{fig:simuCirc}. Note that the exosystem in this case only allows us to generate $x_r$ which is nondecreasing. One can consider other exosystems which generate larger class of reference signals, but this may require adding more elements in the circuit.

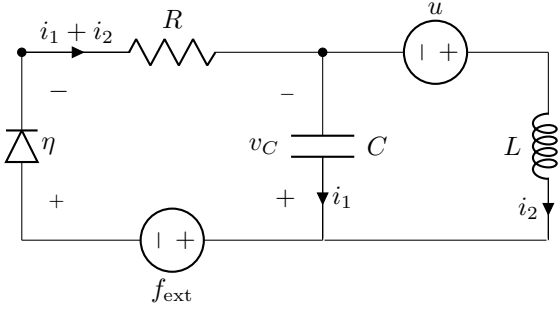
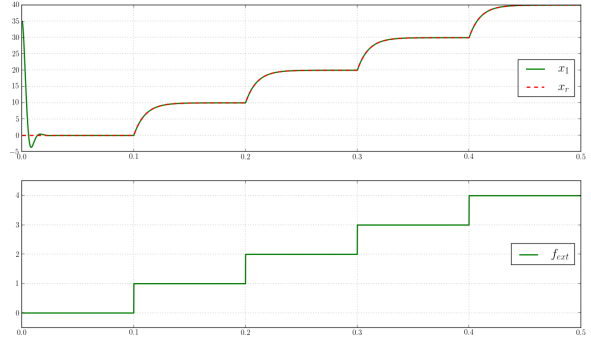
\begin{figure}
\centering
\begin{subfigure}[b]{0.49\linewidth}\centering
\begin{circuitikz}[scale=1, cute inductors]
\ctikzset{voltage/distance from line=0.25}
\ctikzset{bipoles/resistor/height=0.25}
\ctikzset{bipoles/diode/height=0.3}
\ctikzset{bipoles/diode/width=0.3}
\foreach \contact/\y in {1/1,2/5,3/8}
{
\node [contact] (down contact \contact) at (\y,0) {};
\node [contact] (up contact \contact) at (\y,2.5) {};
}

\draw (down contact 3) -- (down contact 2);

\draw (down contact 2) to [voltage source, v^={$f_{\text{ext}}$}] (down contact 1);

\draw (down contact 1) [Do, v_>={$\eta\,$}] to (up contact 1);

\draw (up contact 2) to [C=$C$, i^>=$i_1$, v_<={$v_C$\, }] (down contact 2);

\draw (down contact 3) [L=$L$, i<={$i_2$}] to  (up contact 3);
\draw (up contact 2) to  [R, i_<=$i_1+i_2$, l_=$R$,-*] (up contact 1);
*
\draw (up contact 3) to [voltage source, v_={$u$},-*] (up contact 2);
\end{circuitikz}
\caption{A circuit with diode and external inputs.}
\label{fig:exCircuit}
\end{subfigure}
\begin{subfigure}[b]{0.49\linewidth}\centering
\includegraphics[width=0.95\textwidth]{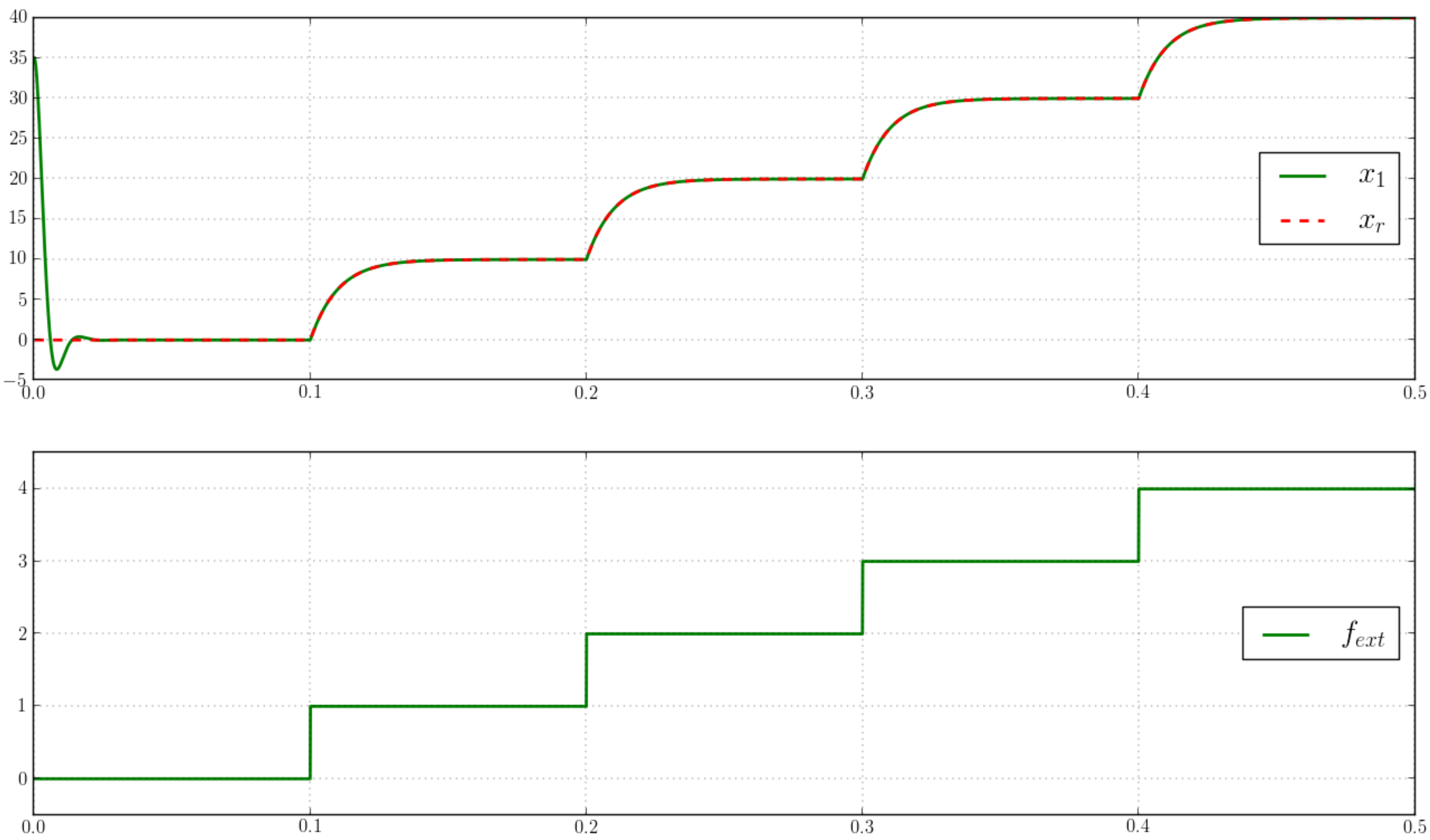}
\caption{Plots showing $f_{\text{ext}}$ and $x_1$ tracking $x_r$.}
\label{fig:simuCirc}
\end{subfigure}
\caption{Regulation of discontinuous function in circuits modeled as EVIs.}
\end{figure}

\end{exam}

\section{Conclusions}\label{sec:conc}
This paper studied the problem of output regulation in a certain class of nonsmooth dynamical systems that are modeled as evolution variational inequalities.
These nonsmooth systems in particular model systems where states are constrained to evolve within some closed, convex and time-varying set.
We first studied the conditions under which there exists a unique solution for such systems.
The classical internal model principle was then used to derive conditions to synthesize a control law that achieves the desired objective.
The analysis were based on using the Lyapunov methods in combination with monotonicity property of the normal cone operator to prove stability.

One could think of extending the results to consider the case when the set-valued mapping $\cS(\cdot)$ is not just time-dependent, but also state-dependent.
This would allow us to model mechanical systems with impacts within our framework but deeper investigation is required for the synthesis of control in such cases.



\appendix


\section{Useful Results from Literature}\label{app:litLemmas}

\begin{thm}[{\cite[Theorem~3]{KunzMont97}}]\label{thm:solMaxMonTV}
Consider the differential inclusion
\begin{equation}\label{eq:sysGenLit}
\dot z(t) \in -\cM(t,z(t)), \quad z(0) \in \dom \cM(0,\cdot),
\end{equation}
where $\cM(t,\cdot):\cH \rightrightarrows \cH$ is a maximal monotone operator for each $t \in [0,T]$. Assume that the following two conditions hold:
\begin{hyp}
\item\label{hyp:minNorm} The least norm element $\eta^0(t,z)$ of the set $\cM(t,z)$ satisfies
\[
\vert \eta^0 (t,z) \vert \le \lambda(t) \, (1 + \vert z \vert), \quad \text{ for } t \in [0,T], z \in \dom \cM(t,\cdot)
\]
for some integrable function $\lambda \in \cL^1([0,T],\R)$.
\item\label{hyp:contDom} There exists an absolutely continuous function $\mu:[0,T] \rightarrow \R$ such that
\[
d_{\svm} (\cM(t_1,\cdot), \cM(t_2,\cdot)) \le \vert \mu(t_1)- \mu(t_2)\vert, \quad \text{ for } t_1,t_2 \in [0,T]
\]
where $d_{\svm} (\cM_1, \cM_2)$ defines a pseudo-distance between two operators $\cM_1$, $\cM_2$ and is given through
\begin{equation}\label{eq:defdsvm}
d_{\svm} (\cM_1, \cM_2) :=\sup \left\{\frac{\langle \eta_1 - \eta_2, z_2 - z_1 \rangle}{1 + |\eta_1| + |\eta_2|}, \eta_i \in \cM_i(z_i), z_i \in \dom(\cM_i), i=1,2 \right\}.
\end{equation}
\end{hyp}
Then for each $z_0 \in \dom \cM(0,\cdot)$, there exists a unique strong solution to the inclusion \eqref{eq:sysGenLit}.
\end{thm}

\begin{lem}[\cite{Robi75}]\label{lem:hausBndK}
Consider the closed convex polyhedral cone $\cK$ defined by $\cK:=\{y \in \R^{d_s} \, \vert \, R \, y \ge 0\}$. For a given matrix $Q$ and vectors $h_1,h_2$, introduce the sets $\cS_1 := \{ x \, \vert \, Qx - h_1 \in \cK\}$ and $\cS_2 := \{ x \, \vert \, Qx - h_2 \in \cK\}$.
Then, there exists a constant $c_\cK> 0$ (depending only on $Q$ and $R$) such that
\[
d_{\Haus}(\cS_1,\cS_2) \le c_\cK \|R\| \, |h_1 - h_2|.
\]
\end{lem}

\begin{thm}[{\cite[Theorem~24.20]{BausComb11}, \cite{BrezHara76}}]\label{thm:sumMon}
Let $\cM_1$ and $\cM_2$ be set-valued monotone operators defined on a Hilbert space $\cH$ such that
\begin{itemize}
\item the sum $\cM_1 + \cM_2$ is maximal monotone, and
\item for $i=1,2$, $\dom \cM_i \times \rge \cM_i \subseteq \dom \cF_i$, where $\cF_i: \cH \times \cH \rightarrow [-\infty,\infty]$ is defined as
\[
\cF_i(x,y) = \langle x,y\rangle - \inf_{(\overline x, \overline y) \in \gph{\cM_i}} \langle x - \overline x, y - \overline y \rangle.
\]
\end{itemize}
Then, it holds that
\begin{gather}
\cl[\rge(\cM_1+\cM_2)] = \cl [\rge\cM_1 + \rge\cM_2] \\
\inn({\rge(\cM_1+\cM_2)}) = \inn[ \rge\cM_1 + \rge\cM_2].
\end{gather}
\end{thm}
The operators satisfying the second condition listed in Theorem~\ref{thm:sumMon} are called $3^*$-monotone. If $\cS$ is a convex set, and $\sigma_{\cS}$ denote its support function, then $\partial \sigma_{\cS}$ is $3^*$-monotone \cite[Example 24.9]{BausComb11}. For $J$ positive definite, the same conclusion follows from \cite[Example~24.11]{BausComb11}.

\section*{Acknowledgements}
The first author would like to thank Kanat Camlibel and Luigi Iannelli for useful discussions related to class of differential inclusions studied in this paper. Several exchanges with them since the earlier version of this article was published at IEEE Conference on Decision and Control at Los Angeles in December 2014, have helped us improve the proofs in Section~\ref{sec:sol}. The authors also thank Dr. Ba Khiet Le for suggesting corrections in the earlier version of the proof of Lemma~\ref{lem:solIncmu0}.

\small
\bibliographystyle{plain}
\bibliography{DiffIncbib}

\begin{thebibliography}{10}

\bibitem{AcarBonn11}
V.~Acary, O.~Bonnefon, and B.~Brogliato.
\newblock {\em Nonsmooth Modeling and Simulation for Switched Circuits},
  volume~69 of {\em Lecture Notes in Electrical Engineering}.
\newblock Springer, Heidelberg, 2011.

\bibitem{AdlyHadd14}
S.~Adly, T.~Haddad, and L.~Thibault.
\newblock Convex sweeping process in the framework of measure differential
  inclusions and evolution variational inequalities.
\newblock {\em Math. Prog. Ser. B}, 48(1):5--47, 2014.

\bibitem{AubiBaye11}
J.-P. Aubin, A.M. Bayen, and P.~Saint-Pierre.
\newblock {\em Viability Theory: New Directions}.
\newblock Springer, 2nd edition, 2011.

\bibitem{AubiCell84}
J.-P. Aubin and A.~Cellina.
\newblock {\em Differential Inclusions: Set-Valued Maps and Viability Theory}.
\newblock Springer-Verlag, 1984.

\bibitem{BausComb11}
H.H. Bauschke and P.L. Combettes.
\newblock {\em Convex Analysis and Monotone Operator Theory in Hilbert Spaces}.
\newblock CMS Books in Mathematics. Springer, 2011.

\bibitem{Brez73}
H.~Br\'ezis.
\newblock {\em Op\'erateurs Maximaux Monotones et Semi-Groupes de Contractions
  dans les Espaces de Hilbert}.
\newblock North-Holland, Mathematics Studies, 1973.

\bibitem{BrezHara76}
H.~Brezis and A.~Haraux.
\newblock Image d'une somme d'operateurs monotones et applications.
\newblock {\em Israel J. Mathematics}, 23(2):165--186, 1976.

\bibitem{Brog04}
B.~Brogliato.
\newblock Absolute stability and the {L}agrange--{D}irichlet theorem with
  monotone multivalued mappings.
\newblock {\em Systems \& Control Letters}, 51:343--353, 2004.

\bibitem{Brog16}
B.~Brogliato.
\newblock {\em Nonsmooth Mechanics: Models, Dynamics and Control}.
\newblock Springer, 3rd edition, 2016.

\bibitem{BrogGoel11}
B.~Brogliato and D.~Goeleven.
\newblock Well-posedness, stability and invariance results for a class of
  multivalued {L}ur'e dynamical systems.
\newblock {\em Nonlinear Analysis Series A: Theory, Methods \& Applications},
  74:195--212, 2011.

\bibitem{BrogGoel13}
B.~Brogliato and D.~Goeleven.
\newblock Existence, uniqueness of solutions and stability of nonsmooth
  multivalued {L}ur'e dynamical systems.
\newblock {\em J. Convex Analysis}, 20(3):881--900, 2013.

\bibitem{CamlIann14}
M.K. Camlibel, L.~Iannelli, and F.~Vasca.
\newblock Passivity and complementarity.
\newblock {\em Math. Prog. Ser. A}, 145(1):531--563, 2014.

\bibitem{CamlSchu15}
M.K. Camlibel and J.M. Schumacher.
\newblock Linear passive systems and maximal monotone mappings.
\newblock {\em Math. Prog. Ser. B}, 157(2):397--420, 2016.

\bibitem{Cortes08}
J.~Cortes.
\newblock Discontinuous dynamical systems: {A} tutorial on solutions, nonsmooth
  analysis, and stability.
\newblock {\em Control Systems Magazine}, 28(3):36--73, 2008.

\bibitem{CottPang92}
R.W. Cottle, J.-S. Pang, and R.E. Stone.
\newblock {\em The Linear Complementarity Problem}.
\newblock Academic Press, 1992.

\bibitem{EdmoThib06}
J.F. Edmond and L.~Thibault.
\newblock {BV} solutions of nonconvex sweeping process differential inclusion
  with perturbation.
\newblock {\em J. Differential Equations}, 226:135--179, 2006.

\bibitem{FaccPang03}
F.~Facchinei and J.-S. Pang.
\newblock {\em Finite-Dimensional Variational Inequalities and Complementarity
  Problems}.
\newblock Springer Series in Operations Research and Financial Engineering.
  Springer, 2003.

\bibitem{Fili88}
A.F. Filippov.
\newblock {\em Differential Equations with Discontinuous Righthand Sides}.
\newblock Springer, 1988.

\bibitem{Francis77}
B.~Francis.
\newblock A linear multivariable regulator problem.
\newblock {\em {SIAM} J. Control \& Optimization}, 15(3):486--505, 1977.

\bibitem{Goeb14}
R.~Goebel.
\newblock Lyapunov functions and duality for convex processes.
\newblock {\em SIAM J. Control \& Optimization}, 52(4):3332--3350, 2014.

\bibitem{Khalil02}
H.K. Khalil.
\newblock {\em Nonlinear Systems}.
\newblock Prentice Hall, 3rd edition, 2002.

\bibitem{KunzMont97}
M.~Kunze and M.D.P.~Monteiro Marques.
\newblock {BV} solutions to evolution problems with time-dependent domains.
\newblock {\em Set-Valued Analysis}, 5:57--72, 1997.

\bibitem{Luen69}
D.G. Luenberger.
\newblock {\em Optimization by Vector Space Methods}.
\newblock John Wiley \& Sons, Inc., New York, 1969.

\bibitem{Mont93}
M.D.P. {Monteiro Marques}.
\newblock {\em Differential Inclusions in Nonsmooth Mechanical Problems: Shocks
  and Dry Friction}, volume~9 of {\em Progress in Nonlinear Differential
  Equations and their Applications}.
\newblock Birkh\"auser, 1993.

\bibitem{Moreau77}
J.J. Moreau.
\newblock Evolution problem associated with a moving convex set in {H}ilbert
  space.
\newblock {\em J. Differential Equations}, 26:347--374, 1977.

\bibitem{PangStew08}
J.-S. Pang and D.E. Stewart.
\newblock Differential variational inequalities.
\newblock {\em Math.~Prog., Ser. A}, 113:345--424, 2008.

\bibitem{PeypSori10}
J.~Peypouquet and S.~Sorin.
\newblock Evolution equations for maximal monotone operators: Asymptotic
  analysis in continuous and discrete time.
\newblock {\em J. Convex Analysis}, 17(3 \& 4):1113 -- 1163, 2010.

\bibitem{Rant15}
A.~Rantzer.
\newblock Scalable control of positive systems.
\newblock {\em European J. Control}, 24:72--80, 2015.

\bibitem{Robi75}
S.M. Robinson.
\newblock Stability theory for systems of inequalities. {P}art i: {L}inear
  systems.
\newblock {\em SIAM J. Numerical Analysis}, 12(5):754--769, 1975.

\bibitem{Rock70}
R.T. Rockafellar.
\newblock {\em Convex Analysis}.
\newblock Princeton University Press, 1970.

\bibitem{RockWets98}
R.T. Rockafellar and R.J.-B. Wets.
\newblock {\em Variational Analysis}, volume 317 of {\em Gundlehren der
  mathematischen Wissenchaften}.
\newblock Springer-Verlag, Berlin, 3rd printing, 2009 edition, 1998.

\bibitem{TanwBrog14a}
A.~Tanwani, B.~Brogliato, and C.~Prieur.
\newblock Stability and observer design for multivalued {L}ur'e systems with
  non-monotone, time-varying nonlinearities and state jumps.
\newblock {\em {SIAM} J. Control \& Optimization}, 52(6):3639--3672, 2014.

\bibitem{TanwBrog15}
A.~Tanwani, B.~Brogliato, and C.~Prieur.
\newblock Stability notions for a class of nonlinear systems with measure
  controls.
\newblock {\em Mathematics of Control, Signals, \& Systems}, 27(2):245--275,
  2015.

\bibitem{TanwBrog16}
A.~Tanwani, B.~Brogliato, and C.~Prieur.
\newblock Observer design for unilaterally constrained {L}agrangian systems:
  {A} passivity-based approach.
\newblock {\em {IEEE} Trans. on Automatic Control}, 61(9):2386--2401, 2016.

\bibitem{TanwBrog14b}
A.~Tanwani, B.~Brogliato, and C.~Prieur.
\newblock On output regulation in state-constrained dynamical systems: {An}
  application to polyhedral case.
\newblock In {\em Proc.~19th {IFAC} World Congress}, pages 1513--1518, August
  2014.

\bibitem{TanwBrog14c}
A.~Tanwani, B.~Brogliato, and C.~Prieur.
\newblock On output regulation in systems with differential variational
  inequalities.
\newblock In {\em Proc.~53rd {IEEE} Conf. on Decision \& Control}, pages
  3077--3082, December 2014.

\bibitem{VascCaml09}
F.~Vasca, M.K. Camlibel, L.~Iannelli, and R.~Frasca.
\newblock A new perspective for modeling power electronics converters:
  Complementarity framework.
\newblock {\em {IEEE} Trans. on Power Electronics}, 24(2):456--468, 2009.

\bibitem{Vlad91}
A.A. Vladimirov.
\newblock Nonstationary dissipative evolution equations in a {H}ilbert space.
\newblock {\em Nonlinear Analysis}, 17:499--518, 1991.

\end{thebibliography}

\end{document}